\documentclass[12pt,oneside]{amsart}
\usepackage{t1enc}
\usepackage[latin2]{inputenc}

\usepackage{amsmath,amsfonts}
\usepackage{amsthm}
\usepackage{amscd}
\usepackage{amssymb}
\usepackage{enumerate}
\usepackage{mathrsfs}
\usepackage{dsfont}

\usepackage{stmaryrd}
\usepackage[T1]{fontenc}

\usepackage{mathtools}

\usepackage{bbm}

\usepackage{wasysym}
\usepackage{verbatim}
\usepackage{fancyhdr}
\usepackage{fancybox}
\usepackage{url}
\usepackage{xcolor}

\newcommand{\bs}{\backslash}
\newcommand{\subs}{\subseteq}
\newcommand{\half}{\frac{1}{2}}

\usepackage[pdfencoding=auto, psdextra]{hyperref}
\hypersetup{colorlinks=true, citecolor=black, linkcolor=black, 
bookmarksopen=true,
}

\newcommand{\ef}{\mathcal{F}}
\newcommand{\G}{\mathcal{G}}

\newcommand{\bc}{\begin{center}}
\newcommand{\ec}{\end{center}}

\DeclareMathOperator{\supp}{supp}

\newcommand{\F}{\mathcal{F}}

\DeclareMathOperator{\cf}{cf}

\newtheorem{thm}{Theorem}[section]
\newtheorem*{thm*}{Theorem}
\newtheorem{lem}[thm]{Lemma}
\newtheorem{prop}[thm]{Proposition}
\newtheorem{cor}[thm]{Corollary}

\theoremstyle{definition}

\newtheorem{prob}[thm]{Problem}

\newtheorem{df}[thm]{Definition}
\newtheorem{exa}[thm]{Example}

\newtheorem{rem}[thm]{Remark}

\title{Non-separable combinatorial Banach spaces}

\author{Piotr Borodulin--Nadzieja}
\address[Piotr Borodulin-Nadzieja]{Mathematical Institute, University of Wroc\l aw \\   pl. Grunwaldzki 2, 50-384 Wroc\l aw, Poland}
\email{pborod@math.uni.wroc.pl}

\author{Miko\l aj Marsy}
\address[Miko\l aj Marsy]{Mathematical Institute, University of Wroc\l aw \\   pl. Grunwaldzki 2, 50-384 Wroc\l aw, Poland}
\email{mikolaj.marsy@math.uni.wroc.pl}

\author{Kamil Ryduchowski}
\address[Kamil Ryduchowski]{Faculty of Mathematics, Informatics and Mechanics, University of Warsaw, ul. Banacha 2, 02-097 Warszawa, Poland}
\address{Institute of Mathematics, Polish Academy of Sciences, ul. \'Sniadeckich 8, 00-656 Warszawa, Poland}
\email{k.ryduchowski@uw.edu.pl}


\begin{document}

\begin{abstract} We study combinatorial Banach spaces, i.e. Banach spaces induced by families of finite sets, of uncountable density. We prove some general theorems, for example we characterize when such spaces contain isomorphic
	copies of long $c_0$ and long $\ell_1$. We discuss the existence of Banach spaces which are (complementedly) universal in the category of Banach spaces with unconditional basis of fixed cardinality. We present some examples, in particular we show
	that the Banach spaces generated by Suslin trees do not contain uncountable equilateral sets.
\end{abstract}

\maketitle

\section{Introduction}

Consider a cardinal number $\kappa$. Let $\mathcal{F}$ be a family of finite subsets of $\kappa$ which is closed under taking subsets and which covers $\kappa$. Define $\| \cdot \|_\mathcal{F} \colon \mathbb{R}^\kappa \to [0,\infty]$ by
\[ \| x \|_\mathcal{F} = \sup_{F\in \mathcal{F}} \sum_{i\in F} |x(i)|. \]

Let $X_\mathcal{F}$ be the completion of $c_{00}(\kappa)$ with respect to the above norm, and let $Y_\mathcal{F} = \{x\in \mathbb{R}^\kappa\colon \| x\|_\mathcal{F} < \infty\}$. Endowed with $\|\cdot\|_\mathcal{F}$, both $X_\mathcal{F}$ and
$Y_\mathcal{F}$ are Banach spaces. The spaces of the form $X_\mathcal{F}$ will be called \emph{combinatorial Banach spaces}.

Combinatorial Banach spaces were studied systematically for $\kappa = \omega$, see e.g. in \cite{ZOO}, \cite{Perfect}.  

Nonseparable Banach spaces defined in the above manner appeared e.g. in the work of Shelah and Steprans (\cite{Shelah-Steprans}), and in an article by Koszmider and Ryduchowski (\cite{KamilKoszmider}), although in the latter the authors considered the
$\ell_2$ norm instead of the $\ell_1$ norm. 

The aim of this article is to study structural properties of nonseparable Banach spaces. We are particularly interested in the question of which (analytic) properties of spaces $X_\mathcal{F}$ can be characterized by (combinatorial) properties of
$\mathcal{F}$ or its closure (in the Cantor topology).

For example, one can consider if a Banach space contains an isomorphic copy of $\ell_1(\lambda)$ or $c_0(\lambda)$, for a cardinal $\lambda$. In the separable case, the property of containing a copy of $\ell_1$ is a typical 'combinatorial' property: a space $X_\mathcal{F}$
contains a copy of $\ell_1$ if and only if there is an infinite set all of whose finite subsets are in $\mathcal{F}$ (i.e. $\overline{\mathcal{F}}$, the closure of $\mathcal{F}$, in the Cantor topology, contains an infinite set). On the other hand, the
property of containing a copy of $c_0$ seems to be impossible, or at least difficult,
to encode in the combinatorial structure of $\mathcal{F}$ (see \cite{ZOO} for a discussion).

In light of the above results, studying these kind of questions in the nonseparable case brings surprises. First, for $\kappa>\omega$ the question if $X_\mathcal{F}$ contains a copy of $c_0(\kappa)$ boils down to an easy combinatorial question
about $\mathcal{F}$ (Theorem \ref{c0-general}).
Second, the characterization of containing a copy of $\ell_1(\kappa)$ turns out, rather unexpectedly, to be connected to the question of existence of measures on $\overline{\mathcal{F}}$ of certain Maharam type (Theorem \ref{l1-general}). 

In \cite{Perfect} the authors investigated families $\mathcal{F}$ induced by colorings: for a function $c\colon [\kappa]^2 \to \{0,1\}$ let $\mathcal{F}_c$ be the family of all finite monochromatic sets in color $1$. The
Banach spaces induced by the families of the form $\mathcal{F}_c$ behave in a quite regular way. For example, the space $X_{\mathcal{F}_c}$ contains a copy of $\ell_1(\kappa)$ if and only if $\overline{\mathcal{F}_c}$ contains a set of size $\kappa$.
In \cite{Perfect}
the authors proved a theorem saying that a coloring $c\colon [\omega]^2 \to \{0,1\}$ is perfect (see Section \ref{sec:Preliminaries} for the definition) if and only if $X^*_{\mathcal{F}_c} = Y_{\mathcal{F}_{1-c}}$, where $X^*_{\mathcal{F}_c}$ is the
space dual to $X_{\mathcal{F}_c}$. We prove a certain generalization of
this result for the nonseparable case  (Corollary \ref{perfect}).

One of our main motivations to study nonseparable combinatorial spaces is that on uncountable cardinals one can find a plethora of interesting combinatorial structures: Aronszajn trees, Suslin trees, Hausdorff gaps, Sierpi\'nski colorings to name
only a few of them living on
$\omega_1$. Some of them may be used, in a natural way, to define families of finite subsets, and so also to define combinatorial Banach spaces. Studies of
separable combinatorial spaces seem to support a meta-mathematical statement saying that that canonical combinatorial objects should induce important Banach spaces. If this is the case, the structures mentioned above should provide a rich family of
possibly important Banach
spaces.

In Section \ref{sec:Suslin} we focus on Suslin trees, probably the most notable creatures in the menagerie of canonical structures in $\omega_1$. Every tree generates very natural families of finite sets: finite chains and finite antichains. Let $X_S$ be the
space generated by the family of finite chains of a Suslin tree. Then $X_S$, a nonseparable Banach space, does not contain either copies of $\ell_1(\omega_1)$, or of $c_0(\omega_1)$.This property is already quite interesting, as if $X$ is a
separable combinatorial space then it has to contain copies of $\ell_1$ or $c_0$. However, in the nonseparable setting spaces with this property exists even in $\mathsf{ZFC}$: if $c\colon [\omega_1]^2 \to \{0,1\}$ is a Sierpi\'nski coloring (see
Section \ref{sec:Examples} for the definition), then $X_c$ does
not contain neither copies of long $\ell_1$, nor of long $c_0$. In Section \ref{sec:Suslin} we find another interesting, and perhaps unexpected, property of the Banach space induced by Suslin tree: it does not contain an uncountable equilateral set
(Theorem \ref{Suslin-equilateral}).
Note that under $\mathsf{MA}(\omega_1)$ every combinatorial Banach space contains an uncountable equilateral set (Proposition \ref{MA-equilateral}). Some constructions of combinatorial Banach spaces without uncountable equilateral sets have been
known before (see e.g. \cite{KamilKoszmider}), but our construction seems to be simpler, and it is generated by a classical object. 

In Section \ref{sec:Universal} we study another canonical structure. In \cite{ZOO} the authors showed that the combinatorial Banach space induced by so-called universal hypergraph on $\omega$ is isomorphic to the Pe\l czy\'nski complementedly universal space. Here,
we show that as long as a universal hypergraph does exist on $\kappa$, then there exists a (combinatorial) Banach space $X$ of density $\kappa$ such that every Banach space with unconditional base of size $\kappa$ embeds in a complemented way in $X$.
Such a space $X$ is a natural generalization of the Pe\l czy\'nski space for higher cardinals. To the best of our knowledge, it has not been studied before. We also show that consistently there are no such spaces of density $\omega_1$: we show it by
iterating $\omega_2$ many times the most natural forcing notion adding a
hereditary family of finite sets covering $\omega_1$ (which is equivalent to the forcing adding $\omega_1$ many Cohen reals). The forcing itself seems to add an interesting Banach space - it does not contain an isomorphic copy of any ground-model Banach space with unconditional basis of size $\omega_1$.

\section{Preliminaries}\label{sec:Preliminaries}

We will deal with families $\mathcal{F}$ of finite subsets of $\kappa$, where $\kappa$ is an infinite cardinal number, which are hereditary (i.e. closed under taking subsets) and which contain $\{\alpha\}$ for each $\alpha < \kappa$. Such a family $\mathcal{F}$ can be seen as a subset of $2^\kappa$ (via characteristic functions) so we can speak about
topological properties of $\mathcal{F}$ and consider the closure $\overline{\ef}$ in the product topology on $2^\kappa$. Since $\ef$ is hereditary, being an element of its closure has a simple combinatorial characterization: $A\subseteq \kappa$ belongs to it if and only if $[A]^{<\omega} \subseteq \mathcal{F}$. Consequently, $\mathcal{F}$ as above is compact if and only if there is no infinite $A\subseteq \kappa$ such that $[A]^{<\omega} \subseteq \mathcal{F}$. 

There is a natural way of generalizing the property characterizing compactness. Call a hereditary family $\mathcal{F}$ of subsets of $\kappa$ \emph{Corson} if $\overline{\mathcal{F}}$ does not contain an uncountable set. Notice that $\mathcal{F}$ is
Corson if and only if $\overline{\mathcal{F}}$ is Corson compact, i.e. it is a compact family which can be embedded in $\Sigma(\mathbb{R}^{\omega_1}) = \{x\in \mathbb{R}^\kappa\colon |\mathrm{supp}(x)| \leq \omega\}$. Similarly, for
an uncountable cardinal $\lambda\leq \kappa$ say that $\mathcal{F} \subseteq [\kappa]^{<\omega}$ is $\lambda$-Corson if $\overline{\mathcal{F}}$ does not contain a set of size $\lambda$. Again, $\mathcal{F}$ is $\lambda$-Corson if and only if
$\overline{\mathcal{F}}$ is $\lambda$-Corson compact, i.e. it can be embedded in $\{x\in \mathbb{R}^\kappa\colon |\mathrm{supp}(x)|<\lambda\}$.  

\begin{prop} Suppose that $\mathcal{F}$ is a hereditary family of subsets of $\omega_1$. Denote $K = \overline{\mathcal{F}}$. Then
	\begin{enumerate}
		\item $K$ is not first countable,
		\item $\mathcal{F}$ is a Corson family iff there exists no continuous surjection $f\colon K \to [0,1]^{\omega_1}$,
		\item $\mathcal{F}$ is a Corson family iff  $K$ is Corson compact.
	\end{enumerate}
\end{prop}

\begin{proof} 
	(1) Notice that $\emptyset$ does not have a countable base.
	\medskip

	(2 $\impliedby$). Suppose that $\ef$ is not a Corson family and that $I\subseteq \kappa$ is a set of size $\omega_1$ from $\overline{\mathcal{F}}$. Let $g\colon K \to \{0,1\}^{I}$ be defined by $g(x) = x|I$. Then $g$ is
	a continuous surjection and $\{0,1\}^I$ can be continuously transformed onto $[0,1]^{\omega_1}$. 
	\medskip

	(3 $\implies$). If $\ef$ is Corson, then  $K\subs \Sigma_\kappa(\{0,1\})$.
	\medskip

	(3 $\impliedby$) Note that Corson compacta cannot be mapped continuously onto $[0, 1]^{\omega_1}.$ By (2 $\impliedby$) if $\ef$ is not a Corson family, there exists a surjection, hence $K$ is not a Corson compact.
    \medskip
    
    (2 $\implies$) If $\ef$ is Corson, then by (3 $\implies$) $K$ is a Corson compact, so such surjection cannot exist.
\end{proof}

We will be interested in the families of finite sets which are generated by colorings, i.e. functions of the form $c\colon [\kappa]^2 \to 2$. For a coloring $c$ let $\mathcal{F}_c$ be the family of all finite 1-homogeneous sets.

For a family $\mathcal{F}$, define \[ \mathcal{F}^\perp = \{ A\in [\kappa]^{<\omega}\colon |A \cap F|\leq 1 \mbox{ for each }F\in \mathcal{F}\}. \]
For example, if $\mathcal{F}$ is a family of finite $1$-monochromatic sets for a coloring, then $\mathcal{F}^\perp$ is a family of finite $0$-monochromatic sets. Also, in this case $\overline{\mathcal{F}}$ is the set of all $1$-monochromatic sets,
and $\overline{\mathcal{F}^\perp}$ is the set of all $0$-monochromatic sets.
In fact, note that $\mathcal{F}^{\perp\perp} = \mathcal{F}$ if and only if $\mathcal{F}$ is a monochromatic set for some coloring (see \cite{Perfect}).

We will say that a (hereditary) family of subsets of $\kappa$ is \emph{monochromatic} if there is a coloring $c\colon [\kappa]^2 \to 2$ such that $\mathcal{F}$ is the family of finite 1-monochromatic subsets of $\kappa$. For a coloring $c\colon
[\kappa]^{2} \to \{0,1\}$ we write $X_c$, $Y_c$ instead of $X_{\mathcal{F}_c}$ and $Y_{\mathcal{F}_c}$. Sometimes it will be more convenient to speak about graphs instead of coloring. Every coloring $c\colon [\kappa]^2 \to \{0,1\}$ naturally generates a graph on $\kappa$ ($\{\alpha, \beta\}$ is an edge iff
$c(\alpha, \beta)=1$) and vice versa. For a graph $G$ by $X_G$ we denote the Banach space induced by the family of finite cliques of $G$. 

A graph is \emph{perfect} if its every finite subgraph (induced by
vertices) has the clique number equal to chromatic number (see \cite{Perfect}). Say that a coloring is perfect if the graph induced by the coloring is perfect. 

We call a family $\F$ \emph{bounded} if there is $n<\omega$ such that $|F|\leq n$ for every $F\in \F$.

Now we want to mention some useful facts about space $X_\ef$ for a hereditary family $\ef$ covering $\omega_1$. For each $x\in X_\mathcal{F}$ we have 
\[
    |\{\alpha\in\omega_1: x(\alpha) \neq 0\}| \leq \omega.
\] 
We call this set \textit{a support of }$x$ (and denote $\supp(x)$). Note, that if support was to be uncountable, then there would be an uncountable $I\subseteq \kappa$ and $r>0$ such that $|x(\alpha)|\geq r$ for each $\alpha \in I$, and then $x$ could not be approximated by elements of $c_{00}(\kappa)$.

If $(e_\alpha)_{\alpha\in I}$ is a base of a Banach space $X$, $x\in X$, and $A\subseteq I$, then by $P_A(x)$ we denote the projection of $x$ onto $A$, i.e. $x'\in X$ such that $x'(\alpha)=x(\alpha)$ if $\alpha \in A$ and $x'(\alpha)=0$ otherwise.

Notice that every $\varphi\in X^*_\mathcal{F}$ is uniquely determined by its values on Shauder basis, therefore $\varphi$ can be treated as an element of $\mathbb{R}^\kappa$.


In fact, there is one more Banach space associated to $\mathcal{F}$ in a natural way: the space of continuous functions on $\overline{\mathcal{F}}$. 

\begin{prop}\label{C(F)}
    Let $\F$ be a hereditary family on $\kappa$. Then $X_\F$ is a closed subspace of $C(\overline{\F})$.
\end{prop}

\begin{proof}
	The most natural mapping $T\colon X_\F \to C(\overline{\F})$, defined by \[ Tx (A) = \sum_{\alpha \in A} x(\alpha), \] works. Since every $x$ has a countable support and is absolutely convergent, $Tx$ yields a well defined function on $\overline{\F}$.
	We shall check that this function is continuous. Pick $\varepsilon > 0$. Then let $x_0$ be an element of $c_{00}(\kappa)$ with $\|x_0 - x\| <\varepsilon.$ Let $S$ be a support of $x_0$. Then pick $A, B \in \overline{\F}$ satisfying $A\cap S = B\cap S$. We have 

	\begin{align*}
            |Tx(A) - Tx(B)| = & \bigg | \sum_{\alpha \in A\Delta B} x(\alpha) \bigg | \leq \sum_{\alpha \in A \Delta B} |x(\alpha)| = \sum_{A\bs B} |x(\alpha)| + \sum_{B\bs A} |x(\alpha)| = \\
			& \leq \sum_{B \bs S} |x(\alpha)| + \sum_{A \bs S} |x(\alpha)| \leq 2 \|x - x_0\|_\mathcal{F} \leq 2\varepsilon.
    \end{align*}
	
	Now we will check that $T$ itself is continuous.	Obviously, we have $\|Tx\|_\infty \leq \|x\|_\mathcal{F}$.
    Now, pick $A \in \overline{\F}$. Let $A_+ = \{\alpha \in A: x(\alpha) > 0\}, A_-= \{\alpha \in A: x(\alpha) < 0\}$. We have \[
    \sum_{\alpha \in A} |x(\alpha)| \leq \bigg |\sum_{\alpha\in A_+} x(\alpha) \bigg | + \bigg | \sum_{\alpha \in A_-} x(\alpha) \bigg | \leq 2 \| Tx\|_\infty,
\]so $\half \| x\|_\mathcal{F} \leq \|Tx\|_\infty.$
\end{proof}

\section{Functionals on combinatorial spaces.}\label{sec:Functionals}

In this section we will prove some basic results concerning spaces dual to $X_\mathcal{F}$'s and the spaces $Y_\mathcal{F}$. 
We will check for which families $\mathcal{F}$ the spaces  $X^*_\mathcal{F}$ and $Y_\mathcal{F}$ contains only vectors with countable supports. Finally, we will prove a generalization of Theorem 4.6 from \cite{Perfect}, saying that $X^*_\mathcal{F}
= Y_\mathcal{G}$ precisely when $\mathcal{F}$ is induced by a perfect coloring $c$ and $\mathcal{G}$ is induced by the coloring $1-c$.

\begin{prop}\label{countable-support-in-*}
Let $c$ be a coloring on $\kappa>\omega$. Then $c$ does not have an uncountable $1$-monochromatic set iff every functional from $X^*_c$ has a countable support.
\end{prop}
\begin{proof}

	Suppose first that there is $I\subseteq \kappa$ of size $\omega_1$ which is $1$-monochromatic for $c$. Then $\varphi = \chi_I \in X^*_c$. 

	Assume now that there is $\varphi\in X_c^*$ and an uncountable set $I\subseteq \kappa$ such that $\varphi(e_\alpha)\neq 0$ for $\alpha\in I$. We may assume that $\|\varphi\|^*_c =1$. By shrinking $I$ and replacing $\varphi$ with $-\varphi$ if necessary we may also assume that there exists $\varepsilon>0$ such that $\varphi(e_\alpha)>\varepsilon$ for $\alpha\in I$.

	Since we assumed that $I$ has no uncountable $1$-monochromatic subsets, by the Erd\H{o}s--Dushnik--Miller theorem there is an infinite $0$-monochromatic $J\subseteq I$. Pick $n>\frac{1}{\varepsilon}$, take $A\subseteq J$ of size $n$ and consider
$x=\chi_A$. Then $\|x\|_c=\|x\|_\infty=1$ and $\varphi(x) > n\varepsilon > 1$, which contradicts the fact that $\|\varphi\|^*_c=1$.
\end{proof}

Let $\ef$ be a family of subsets of $\kappa$ and let $S\subs \kappa$. Then we denote
\[
    \ef | S = \ef \cap P(S).
\]
For arbitrary families, not necessarily associated to a coloring, we can prove a slightly weaker result:

\begin{prop}
	Suppose $\F \subseteq [\kappa]^{<\omega}$ is a family such that for every uncountable $S\subseteq \kappa$ the family $\F^\perp|S$ is not bounded. Then every functional from $X^*_\mathcal{F}$ has a countable support.
\end{prop}
\begin{proof}
	By way of contradiction assume that there is $\varphi\in X_\F^*$ with an uncountable support and that $X^*_\ef$ satisfies condition in the statement. Let $I\subseteq \kappa$ be an uncountable subset of $\supp(\varphi)$. We may assume that $\|\varphi\|^*_\mathcal{F} =1$ and that there exists $\varepsilon>0$ such that $\varphi(e_\alpha)>\varepsilon$ for $\alpha\in I$. Now, pick $n>\frac{1}{\varepsilon}$, take any $n$-element $A\subseteq I$, $A\in \F^\perp$ and consider
    $x=\chi_A$. Then $\|x\|_\F=\|x\|_\infty=1$ and $\varphi(x) > n\varepsilon > 1$, which contradicts the fact that $\|\varphi\|^*_\F=1$.
\end{proof}

\begin{prop}\label{countable-support-in-Y}
Let $c$ be a coloring of $\kappa>\omega$ without uncountable $0$-monochro-matic sets. Then every $x\in Y_c$ has a countable support.
\end{prop}
\begin{proof}
	Assume that there is $x\in Y_c$ with an uncountable support. Let $A\subs \supp(x)$ be such that for some $\varepsilon>0$ we have $x(\alpha)>\varepsilon$ for every $\alpha\in A$. By the Erd\H{o}s--Dushnik--Miller theorem there is an infinite $1$-monochromatic $B\subseteq A$. Then clearly 
    $$\|x\|_c \geq \|P_B x\|_c = \|P_B x\|_1 \geq \sum_{\alpha\in B} x(\alpha) = +\infty,$$
    which yields a contradiction.
\end{proof}

Again, we can prove in a similar manner the following:

\begin{prop}
Let $\F\subseteq [\kappa]^{<\omega}$ be such that $\F$ is not bounded on any uncountable set.
Then every $x\in Y_\F$ has a countable support.
\end{prop}

\begin{rem}
Note that for a monochromatic family $\F\subseteq [\kappa]^{<\omega}$  the following are equivalent: 
\begin{enumerate}
	\item There is a set $S\in [\kappa]^\kappa$ such that $\F|S$ is bounded,
\item There is a 0-monochromatic set of size $\kappa$.
\end{enumerate}
Indeed, $2)\Rightarrow 1)$ is obvious, for $1)\Rightarrow 2)$ note that by the Erd\H{o}s-Dushnik-Miller theorem either there is a $0$-monochromatic $S'\subseteq S$ of size $\kappa$, since there can be no infinite set in the closure of $\F|S$.
\end{rem}

\begin{cor}\label{perfect}
Let $c$ be a perfect coloring of $\kappa > \omega$ without uncountable $1$-monochromatic sets. Then $X_c^* = Y_{1-c}$.
\end{cor}
\begin{proof}
	Take $\varphi\in X_c^*$. By Proposition \ref{countable-support-in-*} it has a countable support $A\subseteq \kappa$. It allows us to view $\varphi$ as a functional on $X_d$, where $X_d$ is a separable combinatorial Banach space induced by
	$d=c|_{[A]^2}$. Since $d$ is a perfect coloring, it follows from \cite[Theorem 4.6]{Perfect} that $\varphi\in Y_{1-d}$, which implies that $\varphi\in Y_{1-c}$.

	Let now $\psi \in Y_{1-c}$. Since $1-c$ admits no uncountable $0$-monochromatic sets, by Proposition \ref{countable-support-in-Y} $\psi$ has a countable support $B\subseteq \kappa$. Let $e=c|_{[B]^2}$. Clearly we may view $\psi$ as an element of $Y_{1-e}$. Since $e$ is perfect, $\psi \in X_e^*$ and its natural extension belongs to $X_c^*$.
\end{proof}

\section{Copies of \texorpdfstring{$\ell_1(\omega_1)$}{l1(omega1)}}

Let $\mathcal{F}$ be a hereditary family of subsets of $\omega_1$. Notice that if $\overline{\mathcal{F}}$ contains an uncountable set, then $\ell_1(\omega_1)$ embeds isometrically to $X_\mathcal{F}$. In this section we will address the question if
we can reverse this implication.

Note that in the separable case we have the following theorem (see \cite{ZOO}):
\begin{prop}\label{l1-separable} Let $\mathcal{F}$ be a hereditary family on $\omega$. Then, the following conditions are equivalent:
	\begin{enumerate}
	\item $\overline{\mathcal{F}}$ contains an infinite set,
	\item $X_\mathcal{F}$ contains a copy of $\ell_1$,
	\end{enumerate}
\end{prop}

As we will see this proposition may be generalized for nonseparable spaces but in a rather unexpected way. Namely, it turns out that what really matters in condition (1) is that $\overline{\mathcal{F}}$ contain a copy of a Cantor set and so it
carries a non-atomic measure.

Let $\mu$ be a Radon measure on a compact space $K$. Say that a family $\mathcal{C}$ of Borel subsets of $K$ is ($\mu$-)\emph{dense} if for every Borel $B\subseteq K$ and every $\varepsilon>0$ there is $C\in \mathcal{C}$ such that $\mu(B \triangle
C)<\varepsilon$. Notice that the family of clopen subsets of $K$ is dense. We say that $\mu$ has \emph{type} $\kappa$ if $\kappa$ is the minimal cardinality of a dense family. A measure of countable type is called \emph{separable}. The type
of a measure is the density of the space $L_1(\mu)$ (see e.g. \cite{Plebanek2024}).

For $\alpha<\kappa$ let $C^i_\alpha = \{x\in 2^{\kappa}\colon x(\alpha)=i\}$.
Let $\mathcal{F}$ be a hereditary family of subsets of $\omega_1$. For each $\alpha<\omega_1$ let $K_\alpha = \overline{\mathcal{F}}\cap C^1_\alpha$. For a finite $F\subseteq\omega_1$ we have
\begin{equation}\label{com}
	F \in \mathcal{F} \iff \bigcap_{\alpha \in F} K_\alpha \ne \emptyset. 
\end{equation}

We will need some basic lemmas from the general measure theory. If $\mu$ is a measure and $f$ is a $\mu$-measurable function, then by $\mu(f)$ we denote $\int f \ d\mu$.

\begin{lem}\label{calculus} Suppose that $\mu$ is a probability measure on $X$ and $f\colon X\to [0,\infty)$ is a $\mu$-measurable function bounded by $M$ and such that $\mu(f)=E$. Then
	\[ \mu(\{t\in X\colon f(t) \geq E/2\}) \geq E/2M. \]
\end{lem}
\begin{proof} Let $X_0 = \{t\in X\colon f(t) \geq E/2\}$. Then $\mu(f) = \mu(f|X_0) + \mu(f|(X\setminus X_0))$ and so 
	\[ E \leq M\mu(X_0) + (E/2)(1-\mu(X_0)). \]
Performing rather straightforward calculations, using the fact that $M\geq M-E/2$, we get
\[ \mu(X_0) \geq E/2M. \]
\end{proof}

It is well known that in the probability space whenever $(A_n)$ is a family of sets of measure $>\delta>0$, then there is $t$ contained in infinitely many $A_n$'s. The following lemma is a strengthening of this fact: we may demand that in this
situation $t$ is contained in a 'non-summable' many $A_n$'s. 

\begin{lem}\label{measure-theory}
Suppose that $\mu$ is a probability measure on $2^\kappa$ and that $(A_n)_n$ is a family of measurable subsets of $2^{\kappa}$ such that $\mu(A_n)>\delta$ for some $\delta>0$ and every $n$. Let $x$ be a bounded sequence of positive reals such that $\sum_i x(i) = \infty$. Then there is $t\in 2^{\kappa}$ such that $\sum_{i\in B_t} x(i) = \infty$, where $B_t = \{n\colon t\in A_n\}$. 
\end{lem}

\begin{proof} Note that $1\geq \delta$. For $n\in \omega$ and $t\in 2^\kappa$ let $T_n = \sum_{i\leq n} x(i)$ and let $$S_n(t) = \sum_{i\leq B_t \cap [0, ..., n]} x(i) = \sum_{i\leq n} x(i)\chi_{A_i}(t).$$
	Fix $n > 0$. Now let $X_n = \{t\in 2^\kappa \colon S_n(t) \geq \delta T_n / 2 \}$. We have $\mu(S_n)> \delta T_n$ and $S_n(x) \leq T_n$. 
	Applying Lemma \ref{calculus} to $S_n$ we get
	\[ \mu(X_n) \geq \mu(\{t\colon S_n \geq \mu(S_n)/2\}) \geq \mu(S_n)/2T_n > \delta/2. \]
	Now let $X = \bigcap_n \bigcup_{m>n} X_m$ and notice that $\mu(X)\geq \delta/2$, in particular it is nonempty. But if $t\in X$, then $\sum_{i\in A_t} x(i) = \infty$.
\end{proof}

This is the main theorem of the section.

\begin{thm}\label{l1-general} Let $\mathcal{F}$ be a hereditary family on $\kappa$ and let $\lambda\leq \kappa$. Suppose that $cf(\lambda) > \omega$. The following conditions are equivalent:
	\begin{enumerate}
		\item The space $\overline{\mathcal{F}}$ carries a measure $\mu$ of type at least $\lambda$.
		\item The space $X_\mathcal{F}$ contains a copy of $\ell_1(\lambda)$.
	\end{enumerate}
\end{thm}
\begin{proof}
	(1) $\implies$ (2). Let $K_\alpha$ be defined as at the beginning of the section. Then there exists a number $\delta > 0$ such that 
    $|\{\alpha <\kappa: \mu(K_\alpha) >\delta\}| \geq \lambda.$
    Suppose this is not the case, then for every $n$ the set $\{\alpha < \kappa: \mu(K_\alpha) > 1/n\}$ has cardinality smaller than $\lambda$. By uncountable cofinality of $\lambda$ we get that $|\{\alpha <\kappa: \mu(K_\alpha) > 0\}| < \lambda$. Hence, the family of clopen subsets of $\overline{\mathcal{F}}$ whose coordinates depend on $N$ is dense and so $\mu$ is of type less than $\lambda$. 
    We claim that the basic sequence $(e_\alpha)_{\alpha \in I}$ in $X_\mathcal{F}$ is equivalent to the standard basis of $\ell_1(\lambda)$. Suppose not. Then we have the following:
	\medskip

	\textbf{Claim.} There is a (necessarily bounded) positive $x$ with the support in $I$ such that $\|x\|_\mathcal{F}<\infty$ but $\|x\|_1= \infty$. 
	\medskip

	Indeed, for each $k\in \omega$ we may find $x_k$ such that $\|x_k\|_\mathcal{F}<1/2^k$, $\|x_k\|_1 = 1$, and the support of $x_k$ is a finite subset of $I$. We may also assume that the supports of $x_k$ are pairwise disjoint. Then $x = \sum_k
	x_k$ is as required.
	\medskip

    By Lemma \ref{measure-theory} we may find $M\subseteq \mathrm{supp}(x)$, $M \in \overline{\mathcal{F}}$ such that $\sum_{i\in M} x(i) = \infty$. But it would mean that $\| x\|_\mathcal{F} = \infty$, a contradiction.
	\medskip

	(2) $\implies$ (1). It follows directly from Proposition \ref{C(F)} and a result due to Pe\l czynski (\cite{Pelczynski68}, see also \cite{Plebanek02}) saying that if a Banach space $X$ contains a copy of $\ell_1(\lambda)$, then $X^*$ contains a copy of
	$L_1(\mu_\lambda)$, where $\mu_\lambda$ is
	the standard measure on $2^\lambda$. However, we will show a direct argument here. Suppose that
	$X_\mathcal{F}$ contains a copy of $\ell_1(\lambda)$. Then there is
	an injective operator $S\colon \ell_1(\lambda) \to C(K)$, where $K = \overline{\mathcal{F}}$ as $C(K)$ contains a closed copy of $X_\mathcal{F}$. So, there is a continuous surjective operator $T\colon C^*(K) \to \ell_1^*(\lambda)$. Thus, there is a
	continuous surjective operator $T\colon M(K) \to \ell_\infty(\lambda)$; here $M(K)$ stands for the space of signed, Radon measures on $K$. The operator
	is defined by $T(\mu)(x) = \int S(x) \ d\mu$. Let $\varphi$ be an  element of $\ell_\infty(\lambda)$ constantly equal 1 and let $\mu$ be such that $T(\mu) = \varphi$. 

	We will show that $\mu$ is of type $\lambda$.  Let $\alpha \ne \beta$. Then  \[ \| S(e_\alpha - e_\beta) \|_{L_1(\mu)} = \int | S(e_\alpha - e_\beta) |
	d\mu = \varphi(|e_\alpha - e_\beta|) = 2 . \]

\end{proof}

\begin{rem}
	Let us comment on the connection between Proposition \ref{l1-separable} and Theorem \ref{l1-general}. In the separable case we have that $\overline{\mathcal{F}}$ contains an infinite set if and only if it is not scattered. By theorem
	of Rudin a compact space is scattered if and only if it does not carry a non-atomic measure. So, we might as well formulate the condition (1) in Proposition \ref{l1-separable} as ``$\overline{\mathcal{F}}$ carries a measure of infinite type''
	thus seeing Proposition \ref{l1-separable} as a particular case of Theorem \ref{l1-general}.
\end{rem}

We say that a Banach space is \emph{weakly Lindel\"of determined}, in short WLD, if there is a linearly dense set $D\subseteq X_\mathcal{F}$ such that for each $\varphi \in X^*$ the set
$$\{ x\in D:\ \varphi(x)\neq 0\}$$ is countable.

For the case $\kappa = \omega_1$ Theorem \ref{l1-general} translates into the following:

\begin{cor} Let $\mathcal{F}$ be a hereditary family on $\omega_1$. Then the space $X_\mathcal{F}$ contains a copy of $\ell_1(\omega_1)$ if and only if the space $\overline{\mathcal{F}}$ carries a non-separable measure $\mu$.
\end{cor}

In fact, for the case $\kappa = \omega_1$ we can gain a broader perspective, thanks to the work of Russo and Hajek, and Arygors and Mercourakis. Following \cite{RussoHajek} we say that Pt\'ak's lemma holds true for $\kappa$ if for every set $S\in [\kappa]^\kappa$ and every hereditary family $\F \subseteq [S]^{<\omega}$ such that
$$ \delta = \inf\{ \|x\|_\F :\ x\in c_{00}(S), x\geq 0, \|x\|_1=1 \} >0 $$
there is a set of size $\kappa$ in $\overline{\F}$.

Let $\lambda$ be the standard product measure on $2^{\omega_1}$ and let $\mathcal{N}_{\omega_1}$ be the $\sigma$-ideal of $\lambda$-null sets. Recall that $\mathrm{cov}(\mathcal{I})$, for an ideal $\mathcal{I}$ on $X$, is the minimal cardinality of a
covering of $X$ into sets from $\mathcal{I}$.

\begin{thm}\label{l_1} Let $\mathcal{F}$ be a hereditary family of subsets of $\omega_1$. The following are equivalent:
	\begin{enumerate}
	\item $\mathrm{cov}(\mathcal{N}_{\omega_1}) > \omega_1$,
	\item $\mathcal{F}$ is Corson iff $X_\mathcal{F}$ does not contain an isomorphic copy of $\ell_1(\omega_1)$, 
    \item $\mathcal{F}$ is Corson iff $X_\mathcal{F}$ is  WLD,
    \item Pt\'ak's lemma holds true for $\omega_1$.
	\end{enumerate}
\end{thm}
\begin{proof} (2) $\implies$ (1). Let $(N_\alpha)_{\alpha<\omega_1}$ be a family of subsets of $\{0,1\}^{\omega_1}$ such that $\lambda(N_\alpha)=0$ and $\bigcup_{\alpha<\omega_1} N_\alpha
	= 2^{\omega_1}$. For each $\alpha<\omega_1$ choose a compact set $K_\alpha \subseteq 2^{\omega_1}\setminus \bigcup_{\xi < \alpha} N_\xi$ whose measure is bigger than 1/2. Define $\mathcal{F} = \{I\in [\omega_1]^{<\omega}\colon
	\bigcap_{\alpha \in I} K_\alpha \ne \emptyset \}$. Then $\mathcal{F}$ is Corson, yet by the same argument argument as in the proof of Theorem \ref{l1-general} it is isomorphic to $\ell_1(\omega_1)$. 
	\medskip

	(1) $\implies$ (2). It follows by Theorem \ref{l1-general} and by a theorem of Kunen and van Mill (\cite{KunenMill}) saying that a Corson compact space $K$ carries a non-separable measure if and only if $\mathrm{cov}(\mathcal{N}_{\omega_1}) = \omega_1$;
	\medskip

	(2) $\iff$ (3) is proved in \cite[Theorem 1.7]{Argyros93}	
	\medskip
	
	(2) $\implies$ (4): Let $\F$ be a family such that 
	$$ \delta = \inf\{ \|x\|_\F :\ x\in c_{00}(\omega_1), x\geq 0, \|x\|_1=1\} >0.$$ Then for every $x\in c_{00}(\kappa)$, $x\geq 0$ we have
	$$\delta\|x\|_1 \leq \|x\|_\F \leq \|x\|_1,$$
	and consequently, since the standard basis of $X_\F$ is unconditional with constant 1, we have it for every $x\in c_{00}(\kappa)$. Therefore $\|\cdot\|_\F$ and $\| \cdot\|_1$ are equivalent and by (3) $\F$ is not Corson, so there is an uncountable set in $\overline{\F}$.
	\medskip
	
	(4)$\implies$ (2): Suppose there is a copy of $\ell_1(\omega_1)$ in $X_\F$. Then by a result of Drewnowski (see \cite{Drewnowski}) there is an uncountable set $\Gamma\subseteq \omega_1$ such that $(e_\alpha)_{\alpha\in \Gamma}$ is equivalent to the standard basis of $\ell_1(\Gamma)$. It follows that clearly $\F|\Gamma$ satisfies 
	$$\inf\{ \|x\|_\F :\ x\in c_{00}(\Gamma), x\geq 0, \|x\|_1=1\} >0$$
	and therefore, since Pt\'ak's lemma holds true for $\omega_1$, $\F$ is not Corson.
\end{proof}

It turns out that the monochromatic families behave in a much more regular way.

\begin{prop} If $\mathcal{F}$ is monochromatic, then the following are equivalent:
	\begin{enumerate}
	\item $\mathcal{F}$ is Corson,
	\item $\ell_1(\omega_1)$ does not embed to $X_\mathcal{F}$.
	\end{enumerate}
\end{prop}
\begin{proof}
	(2) $\implies$ (1) is obvious.
\medskip
	
(1) $\implies$ (2). Suppose that $X_\mathcal{F}$ contains an isomorphic copy of $\ell_1(\omega_1)$. By Theorem \ref{l_1} the space $X_\mathcal F$ is not WLD. Since the set $\{e_\alpha:\ \alpha<\omega_1\}$ is linearly dense in $X_\mathcal F$ and $X_\mathcal F$ is not WLD, it means there is a functional $\varphi\in X_\mathcal F^*$
with an uncountable support. By Proposition \ref{countable-support-in-*} this means that there is an uncountable $B\subseteq \kappa$ with $[B]^2\subseteq \mathcal F$ and so $\mathcal{F}$ is not Corson.
\end{proof}

Let us comment what happens in the case when $\kappa>\omega_1$. Recall that $\kappa$ is a \emph{caliber of measure algebras} if for every family $(A_\alpha)_{\alpha<\kappa}$ of $\mu_\kappa$-positive subsets of $2^\kappa$, there is $x\in 2^\kappa$ such that the set
$\{\alpha\colon x\in A_\alpha\}$ is of size $\kappa$. Notice that the family used in the first part of the proof of Theorem \ref{l_1} witnessed that, consistently,
$\omega_1$ is not a caliber of measure algebras. The following theorem is the main result in \cite{Plebanek97}:

\begin{thm}(Plebanek) If $\kappa \geq \mathrm{cf}(\kappa)>\omega_1$, the following are equivalent for a compact space $K$:
	\begin{enumerate}
		\item if $K$ carries a measure of type $\kappa$, then $K$ can be mapped continuously onto $[0,1]^\kappa$.
	\item $\kappa$ is a caliber of measure algebras. 
	\end{enumerate}
\end{thm}

Examining the proof of the above theorem, one can easily see that if $\kappa$ is not caliber of measure algebras, then there is a $\kappa$-Corson compact space with a measure of type $\kappa$ (see also \cite[Theorem 12.5]{Zakrzewski}).

\begin{thm} If $\kappa \geq \mathrm{cf}(\kappa)>\omega_1$, the following are equivalent for a family $\mathcal{F} \subseteq [\kappa]^{<\omega}$
	\begin{enumerate}
		\item $\ell_1(\kappa)$ can be embedded into $X_\mathcal{F}$ if and only if $\mathcal{F}$ which is not $\kappa$-Corson,
		\item $\kappa$ is a caliber of measure algebras.
	\end{enumerate}
\end{thm}

Note that if $\kappa$ is such that $\lambda^\omega < \kappa$ for each $\lambda<\kappa$, then $\kappa$ is a precaliber of measure algebras (and so, e.g. $\mathfrak{c}^+$ is a precaliber of measure algebras). On the other hand it is consistent that
$\omega_2$ is not a caliber of measure algebras. For the further discussion, see \cite{Plebanek97}.

\section{Copies of long \texorpdfstring{$c_0$}{c0}}

First, notice that it is rather easy to recognize for which families $\mathcal{F}$ the space $X_\mathcal{F}$ is isomorphic to $c_0$ regardless of whether we are dealing with the separable or non-separable case.

\begin{prop}\label{prop:c0} Let $\kappa$ be an infinite cardinal and let $\mathcal{F} \subseteq [\kappa]^{<\omega}$. A space $X_\mathcal{F}$ is isomorphic to $c_0(\kappa)$ if and only if $\mathcal{F}$ is bounded. Consequently, $X_G$ is
	isomorphic to $c_0$ if and only if $G$ is $K_n$-free for some $n\in \omega$ (a graph is $K_n$-free if it does not contain a clique of size $n$).
\end{prop}
\begin{proof}
	The only nontrivial part is to prove that if $X_\mathcal{F}$ is isomorphic to $c_0(\kappa)$, then $\sup\{|F|\colon F\in \mathcal{F}\}$ is finite. Suppose otherwise. Then there is a countable $I\subseteq
	\kappa$ such that $\sup\{|F|\colon F\in \mathcal{F}|I \} = \infty$, where $\mathcal{F}|I = \{F\cap I\colon F\in \mathcal{F}\}$. The space $X_{\mathcal{F}|I}$ is a closed complemented subpace of $X_\mathcal{F}$ which itself is isomorphic to
	$c_0(\kappa)$, and so $X_{\mathcal{F}|I}$ is isomorphic to $c_0$ (see \cite{Granero}). By the uniqueness of the base of $c_0$ (\cite{Zippin}, see also \cite[page 104]{ZOO}) we get a contradiction.
\end{proof}

The following theorem characterizes when a combinatorial Banach space contains an isomorphic copy of $c_0(\lambda)$. Notice that this time it does not work for countable $\lambda$.

\begin{thm}\label{c0-general}
Let $\lambda \leq \kappa$ be infinite cardinals with $\cf(\lambda)>\omega$ and let $\mathcal{F} \subseteq [\kappa]^{<\omega}$ be a hereditary family covering $\kappa$.
The following are equivalent:
\begin{enumerate}[(1)]
\item There is an isomorphic embedding $T\colon c_0(\lambda) \to X_\mathcal F$,
\item There is an operator $T\colon c_0(\lambda) \to X_\mathcal F$ with range of density $\lambda$,
\item There is $S\in [\kappa]^\lambda$ and $n<\omega$ such that $\mathcal F|S \subseteq [S]^{\leq n}$.
\end{enumerate}
\end{thm}
\begin{proof}
$(1)\implies (2)$ is obvious. For $(3)\implies (1)$ just notice that for $x\in c_{00}(S)$ we have $\|x\|_\infty \leq \|x\| \leq n\|x\|_\infty$. We will prove $(2)\implies (3)$.
Let $T\colon c_0(\lambda) \to X_\mathcal F$ by an operator with range of density $\lambda$. We may assume that $\|T\|=1$. For $\lambda$-many $\alpha<\lambda$ the value $Te_\alpha$ is non-zero, therefore we may assume that this is the case for all $\alpha<\lambda$.
\medskip

\textbf{Claim.} The family $(\supp(Te_\alpha))_{\alpha<\lambda}$ is point-countable.
\medskip

Suppose this is not the case and find $\gamma<\kappa$ and $I\in [\lambda]^{\omega_1}$ such that for $\alpha\in I$ $$Te_\alpha(\gamma) \neq 0.$$
By passing to a subset we could further assume that there is $\varepsilon>0$ such that $|Te_\alpha(\gamma)|\geq \varepsilon$ and that for every $\alpha, \beta \in I$ the values $Te_\alpha(\gamma)$ and $Te_\beta(\gamma)$ have the same sign. Now pick $n<\omega$ large enough so that $n\varepsilon>1$ and distinct $\alpha_1, \alpha_2, \ldots, \alpha_n \in I$ and note that
$$\|e_{\alpha_1} + e_{\alpha_2} + \ldots + e_{\alpha_n} \|_\infty =1$$
but
$$\|T(e_{\alpha_1} + e_{\alpha_2} + \ldots + e_{\alpha_n})\|\geq \|T(e_{\alpha_1} + e_{\alpha_2} + \ldots + e_{\alpha_n})\|_\infty \geq |T(e_{\alpha_1} + e_{\alpha_2} + \ldots + e_{\alpha_n})(\gamma)| = $$ $$|Te_{\alpha_1}(\gamma)| + \ldots + |Te_{\alpha_n}(\gamma)| > n\varepsilon>1,$$
which yields a contradiction.
\medskip

We may therefore assume that $(\supp Te_\alpha)_{\alpha<\lambda}$ are pairwise disjoint and, by further passing to a subset of $\lambda$ of size $\lambda$, that there is $n<\omega$ such that $\|Te_\alpha\|_\infty > 1/n$ for every
$\alpha<\lambda$. For every $\alpha<\lambda$ pick $\gamma_\alpha < \kappa$ such that $|Te_\alpha(\gamma_\alpha)| > 1/n$.

We claim that for $$S=\{\gamma_\alpha:\ \alpha<\lambda\}$$ we have that $\F|S \subseteq [S]^{\leq n}$.

Indeed, take any $F\in \F|S$, $F=\{\gamma_\alpha:\ \alpha\in A\}$ and notice that
$$ \frac{1}{n}|F| = \| \frac{1}{n} \chi_F\| \leq \| P_F\Big( \sum_{\alpha\in A} Te_\alpha \Big) \| \leq \| \sum_{\alpha\in A} Te_\alpha \| = \| T(\sum_{\alpha\in A} e_\alpha) \| \leq 1$$
and therefore $|F|\leq n$.
\end{proof}

Assuming regularity of $\lambda$ we can include $Y_\F$ to the picture.
\begin{prop}
Let $\kappa, \lambda$ be uncountable cardinals with $\cf(\lambda)= \lambda$ and let $\mathcal{F} \subseteq [\kappa]^{<\omega}$ be a hereditary family covering $\kappa$.
The following are equivalent:
\begin{enumerate}[(1)]
\item There is an isomorphic embedding $T\colon c_0(\lambda) \to X_\mathcal F$,
\item There is an operator $T\colon c_0(\lambda) \to X_\mathcal F$ with range of density $\lambda$,
\item There is $S\in [\kappa]^\lambda$ and $n<\omega$ such that $\mathcal F|S \subseteq [S]^{\leq n}$,
\item There is an isomorphic embedding $T\colon \ell_\infty(\lambda) \to Y_\mathcal F$,
\item There is an element $y\in Y_\F$ with $|\supp(y)| = \lambda.$
\end{enumerate}
\end{prop}
\begin{proof}
(5)$\implies$ (3). Let $S_\varepsilon = \{\alpha: |y(\alpha)| > \varepsilon\}$. There exists an $\varepsilon > 0$ with $|S_\varepsilon| \geq \lambda.$ This set works.
\medskip

(3) $\implies$ (5). We have $\|\chi_S\|_\F \leq n$.
\medskip

(3) $\implies$ (4). Take mapping given by $T(e_\alpha) = \chi_{\{S[\alpha]\}}$, where $S[\alpha]$ is the $\alpha$-th element of $S$. This mapping has norm $\leq n$.
\medskip

(4)$\implies$(5) or (3). By a similar argument as in the proof of the previous proposition we can get that $\{\supp (Te_\alpha): \alpha < \lambda\}$ is point countable. Now, we may assume that $|\supp(Te_\alpha)|< \lambda$ (otherwise we get $(5)$). Then for every $\alpha$ the set $\{\xi<\lambda: \supp(Te_\alpha) \cap \supp (Te_\xi) \neq \emptyset\}$ has cardinality less than $\lambda$. So by regularity of $\lambda$ there exists $I \subs \lambda$ with $|I|=\lambda$ and \[
\forall \alpha, \beta \in I \, \supp(Te_\alpha) \cap \supp (Te_\beta) = \emptyset,
\] Now we can repeat argument from previous proposition to prove $(3)$.
\medskip

\end{proof}
\begin{rem}
	The implication $(2)\implies (3)$ follows easily from a result of Drewnowski (\cite[Theorem 2]{Drewnowski})
\end{rem}

\begin{cor}
Let $\kappa, \lambda$ be infinite cardinals with $\cf(\lambda)>\omega$ and let $\mathcal{F} \subseteq [\kappa]^{<\omega}$ be monochromatic.
The following are equivalent:
\begin{enumerate}[(1)]
\item There is an isomorphic embedding $T\colon c_0(\lambda) \to X_\mathcal F$.
\item There is an operator $T\colon c_0(\lambda) \to X_\mathcal F$ with range of density $\lambda$
\item There is a $0$-monochromatic $S\subseteq \kappa$ of size $\lambda$.
\item There is an isometrically isomorphic embedding $T\colon c_0(\lambda) \to X_\F$.
\end{enumerate}
\end{cor}
\begin{proof}
	Note that if $S\in [\kappa]^\lambda$ satisfies $\mathcal F|S \subseteq [S]^{\leq n}$, then there is $S'\in [S]^\lambda$ such that $\mathcal F|S' \subseteq [S']^{\leq 1}$ -- otherwise by the Erd\H{o}s--Dushnik--Miller theorem we would find an infinite $A\subseteq S,$ $A\in \overline \F$.
\end{proof}

\section{Examples}\label{sec:Examples}

This section can be treated as a kind of 'commercial break', in which we present, without proving anything, some examples of combinatorial structures which have been used or which perhaps may be used to define interesting Banach spaces. We believe that this is just a few of many compelling examples. 

Suppose that $\leq$ is a partial order on $\kappa$. Let $c\colon [\kappa]^2 \to 2$ be the coloring given by $c(x,y) = 1$ iff $x$ and $y$ are comparable. Then $\mathcal{F}_c$ consists of finite $\leq$-chains (and $\mathcal{F}^\perp$ of finite
$\leq$-antichains). The coloring $c$ is perfect (see \cite{Perfect}). Also, note that $\| \cdot \|_c = \| \cdot \|_{\mathrm{finite\ chains}}$ and $\| \cdot \|_{1-c} = \| \cdot \|_{\mathrm{finite\ antichains}}$.

\begin{exa}[Suslin tree] Let $S\subseteq \omega_1$ be a Suslin tree and let $c$ be the \emph{comparability} coloring. By Corollary \ref{perfect} we have $X^*_c = Y_{1-c}$. The space $X_c$ does not contain neither $\ell_1(\omega_1)$ nor $c_0(\omega_1)$.
	We will study $X_c$ in Section \ref{sec:Suslin}.
\end{exa}

\begin{exa}[Sierpi\'nski coloring]  Let $A\subseteq \mathbb{R}$ be of size $\omega_1$ and fix $\preceq$, a well-ordering of $A$. Define $c\colon [A]^2 \to 2$: for $x<y$ let $c(x,y)=1$ iff $x\preceq y$. Then there are neither uncountable $1$-monochromatic not
	uncountable $0$-monochromatic sets, and so $X_c$ does not contain neither $\ell_1(\omega_1)$ nor $c_0(\omega_1)$. Again, $c$ is perfect (as it is generated by a partial order), and so, by Corollary \ref{perfect} we have $X^*_c = Y_{1-c}$. In \cite{Perfect} the authors considered countable versions of such colorings (in
	this case we take an infinite $A\subseteq \mathbb{Q}$, fix a well-ordering of $A$ and proceed as above). It turned out that despite there are many 'non-isomorphic' countable Sierpi\'nski colorings, all the combinatorial Banach spaces induced by them are
	isomorphic. Probably in the uncountable case this is no longer true.
\end{exa}

Note that Sierpi\'nski colorings exist in $\mathsf{ZFC}$ and so we always have a combinatorial Banach space of density $\omega_1$ without copies of $\ell_1(\omega_1)$ and $c_0(\omega_1)$. However, consistently there are uncountable cardinals without
this property.
Recall that $\kappa$ is a \emph{Ramsey cardinal} if for each coloring $c\colon [\kappa]^{<\omega} \to \{0,1\}$, there is a monochromatic set of size $\kappa$. It follows that if $\mathcal{F} \subseteq
	[\kappa]^{<\omega}$, then $X_\mathcal{F}$ contains $\ell_1(\kappa)$ or $c_0(\kappa)$. Also, $\kappa$ is \emph{weakly compact} if for each coloring $c\colon [\kappa]^2 \to \{0,1\}$, there is a monochromatic set of size $\kappa$. Thus, if $\kappa$ is
	weakly compact, then each $X_c$ contains a copy of $\ell_1(\kappa)$ or a copy of $c_0(\kappa)$.

	\begin{exa}[Hungarian graphs] In the literature, especially in articles of Erd\H{o}s, Hajnal, Mat\'e, Komj\'ath, one can find a lot of peculiar examples of colorings and graphs, which may be used to for generating interesting Banach spaces. We present here
		two such examples. A graph $G$ on $\omega_1$ is called \emph{Hajnal-Mat\'e} (HM, in short) if the chromatic number of $G$ is $\omega_1$ but for each $\alpha<\omega_1$ the set $\{\eta<\alpha\colon \{\eta, \alpha\} \in E(G)\}$ is finite. Here, $E(G)$
		denotes the set of edges of $G$. Under $\Diamond$ there
		are HM graphs which are even triangle free. Under Martin's Axiom and negation of $\mathsf{CH}$, there are no HM graphs (see \cite{HajnalMate}).
	Because of Theorem \ref{prop:c0} if $G$ is a triangle-free HM graph, $X_G$ is isomorphic to $c_0(\omega_1)$, so it does not produce any new object in the isomorphism category. However, perhaps this renorming of $c_0(\omega_1)$ possesses some
	interesting properties. 

	A graph on (an uncountable cardinal) $\kappa$ is \emph{Erd\H{o}s-Hajnal} if it has a chromatic number $\kappa$ but each of its subgraphs of size less than $\kappa$ has countable chromatic number. An example of an Erd\H{o}s-Hajnal graph on
	$\omega^{\omega_1}$ is the following: connect $f, g\in \omega^{\omega_1}$ by an edge if $\{\alpha\colon f(\alpha)=g(\alpha)\}$ is countable. Again, $X_G$, for $G$ being Erd\H{o}s-Hajnal, may have interesting properties, particularly from the point of view of its subspaces of density less than $2^{\omega_1}$. 
\end{exa}

\begin{exa}[Shelah families] In \cite{Shelah-Steprans} and \cite{Shelah88} the authors considered a hereditary family $\mathcal{F}$ with the property which in our language can expressed in the following way: for every sequence $\{(\xi_\alpha,
	\beta_\alpha)\colon \xi_\alpha \ne \beta_\alpha < \omega_1,
\alpha<\omega_1\}$ and every $k\in \omega$ there is $X \in [\omega_1]^k$ such that $\{\xi_\alpha\colon \alpha\in X\} \in \mathcal{F}$ and $\{\beta_\alpha\colon \alpha \in X\} \in \mathcal{F}^\perp$. The combinatorial Banach spaces induced by this
family is one of the examples of Banach spaces with \emph{few} operators. In the case of \cite{Shelah-Steprans} it particularly means that every operator on the resulting space is a sum of a multiple of identity and an operator with separable range.
\end{exa}

\begin{exa}[compact families] Compact families are particularly important in the case of separable combinatorial Banach spaces. In fact, most of the combinatorial Banach spaces studied in the literature was induced by compact families. The canonical
	example here is the Schreier family $\mathcal{S}$: $S\subseteq \omega$ belongs to $\mathcal{S}$ if $|S|\leq \min S+1$. For $\kappa > \omega$ there is no obvious counterpart of Schreier family. In  \cite{LargeArgyros} Argyros et al consider a
	hereditary family $\mathcal{F}$ on $\omega_1$ which is compact and which is additionally $\alpha$-large, i.e. for every infinite $X\subseteq \omega_1$, the Cantor-Bendixson rank of $\mathcal{F}|X \colon = \{ F\cap X\colon F\in \mathcal{F}\}$ is at least
$\omega^\alpha+1$ (equivalently $ \mathcal{F}|X$ contains a copy of Schreier family of rank at least $\alpha$). 

The authors of \cite{LargeArgyros} used this family to define Banach spaces but in a different spirit than that of our paper: they construct rather a version of Tsirelson space. 

There is also a similar concept: a family $\mathcal{F}$ is \emph{large} if for every $k\in \omega$ and an infinite set $X\subseteq \omega_1$ there is $F\in \mathcal{F}$, $|F|>k$ and $F\subseteq X$.
Note that there is a compact, hereditary, large family on $\kappa$ if and only if $\kappa$ is not $\omega$-Erd\H{o}s cardinal. 

It is worth mentioning that there is a famous open problem if one can define a family of properties similar to Schreier family on $\omega_1$: Fremlin in \cite{Fremlin-DU} asked if there is a family $\mathcal{F}\subseteq [\omega_1]^{<\omega}$ which is compact and
$1/2$-filling, i.e. such that for each $A\in [\omega_1]^{<\omega}$ there is $F\in \mathcal{F}$ such that $|F|>|A|/2$.

\end{exa}

\section{Universal spaces}\label{sec:Universal}

In this section we will show that whenever there is a universal hypergraph on $\kappa$, there is a combinatorial Banach space $X$ containing a complemented copy of every Banach space with unconditional basis of size $\kappa$. Recall that a Banach
space $X$ contains a \emph{complemented} copy of a Banach space $Y$ if there is a Banach space $Z$ such that $X$ is isomorphic to $Y \oplus Z$, or, equivalently, if there is a continuous linear operator $S\colon X \to X$ such that $S^2=S$ and
$S[X]$ is isomorphic to $Y$. 

In \cite{universalPelczynski} Pe\l czy\'nski gave an example of a Banach space with countable unconditional basis which contains a complemented copy of every Banach space with countable unconditional basis, and moreover he showed that every space
with this property is isomorphic to his example. This space is known in the literature as the Pe\l
czy\'nski complementedly universal space. In what follows, by $\kappa$-Pe\l czy\'nski space we will call the
natural generalization of this notion for higher cardinals: a space $X$ is $\kappa$-Pe\l czy\'nski if it has an unconditional basis of size $\kappa$ and it contains a complemented copy of every Banach space with unconditional basis of size $\kappa$.

For a Banach space $X$, by $B(X)$ we mean the unit closed ball in $X$. We say that an unconditional basis of $(e_\alpha)_{\alpha <\kappa}$ is \emph{monotone} if for any sequences $(a_\alpha), (b_\alpha)$ with $|b_\alpha| \leq |a_\alpha|$ for each
$\alpha$ we have $$\|\sum b_\alpha e_\alpha\| \leq\|\sum a_\alpha e_\alpha\|.$$

\begin{thm}\label{complemented2} For every Banach space $X$ with a monotone unconditional basis of size $\kappa$ there is a hereditary family $\mathcal{F}$ on $\kappa$ such that $X_\mathcal{F}$ contains a complemented isometric copy of $X$. 
\end{thm}

\begin{proof}
	First, notice that if $X$ is a Banach space with unconditional basis $(e_\alpha)_{\alpha<\kappa}$, then $X^*$ can be treated as a set of sequences of length $\kappa$ and if $\sigma \in X^*$ and $x\in X$, then $\sigma(x) = \langle \sigma, x\rangle
	\colon = \sum_{\alpha<\kappa}
	\sigma(\alpha) x(\alpha)$. The set of functionals is norming, in the sense that
	\begin{equation}
		\|x\| = \sup\{ \langle \sigma, x\rangle\colon \|\sigma\|^* = 1\}. \tag{$\star$}
\end{equation}
Notice that we do not have to take all $\sigma$'s. It is enough to consider those with finite supports and values in a fixed dense subset of $[-1,1]$. In what follows as this fixed subset we will take the set of dyadic rationals, i.e. numbers of the form
$k/2^n$, for $n\in \omega$ and $k\in [-2^n, 2^n]$.

Instead of $\kappa$ we will work with $\Omega = \omega \times \kappa$. Let $I_\alpha = \omega \times \{\alpha\}$. 
	For $\alpha<\kappa$ let $\mu_\alpha$ be the measure on $\Omega$ defined by \[ \mu_\alpha = \sum_n \delta_{\langle n, \alpha\rangle}/2^{n+1}.\] 
Notice that $\mu_\alpha$ is supported on $I_\alpha$. 

Define 
\[ \mathcal{F} = \{ F \in [\Omega]^{<\omega}\colon (\mu_\alpha(F))_\alpha \in B(X^*) \}. \]
Let $x_\alpha(\langle i, \xi\rangle) = 1/2^{i+1}$ if $\xi = \alpha$ and $x_\alpha(\langle i, \xi\rangle) = 0$ if $\xi \ne \alpha$. Then $x_\alpha \in X_\mathcal{F}$ and $\|x_\alpha\|_\mathcal{F} = 1$ for each $\alpha$.

We will show that $T\colon X \to X_\mathcal{F}$ induced by $T(e_\alpha) = x_\alpha$ is an isometric embedding. Let $y \in X$. Since $(e_\alpha)_\alpha$ is monotone, we may assume that $y$ tests positive on all basis functionals. Moreover, as we have
mentioned above, we can
determine its norm (in the sense of ($\star$)) by using only finitely supported functionals with dyadic rational values. Notice that any such functional can be represented as $(\mu_\alpha(F))_\alpha$ for some $F\in \ef$. Hence we get  

$$\| Ty \|_\mathcal{F} = \| \sum_{\alpha < \kappa} y(\alpha) x_\alpha \|_\mathcal{F} = \sup_{F\in \mathcal{F}} \sum_{\alpha <\kappa} \mu_\alpha(F) \cdot y(\alpha) = \| y \|.$$ 
The second equality follows from that fact that 
\[ \sum_{\langle n,\beta\rangle \in F} \sum_{\alpha<\kappa} x_\alpha(n,\beta) = \sum_{\alpha<\kappa} \mu_\alpha(F). \]

We will now show that $T[X]$ is complemented in $X_\mathcal{F}$. Define $S\colon X_\mathcal{F} \to \mathbb{R}^\Omega$ by 
\[ S(x)(n, \alpha) = \frac{1}{2^{n+1}}\int x \ d \nu_\alpha, \]
where $\nu_\alpha$ is a counting measure supported on $I_\alpha$.

Notice that $S$ is linear and $S(x_\alpha) = x_\alpha$ for each $\alpha$. Hence, $S$ is identity on $T[X]$ and $T[X] \subseteq S[X_\mathcal{F}]$. Also, it is not hard to see that $S^2 = S$. It remains to show that $S$ is continuous. It will follow that
$S[X_\mathcal{F}] \subseteq X_\mathcal{F}$ and, in fact, $S[X_\mathcal{F}] \subseteq T[X]$ (if $x\in X_\mathcal{F}$, then  $y: = (\int x \ d\mu_\alpha)_\alpha \in X$, and $S(x) = T(y)$).

Let $x\in X_\mathcal{F}$. We will show that $\| S(x) \|_\mathcal{F} \leq \| x \|_\mathcal{F}$. Indeed, fix $\alpha<\kappa$ and a finite set $F_\alpha \subseteq I_\alpha$. There is a set $G_\alpha \subseteq I_\alpha$ such that
$|G_\alpha| \leq
|F_\alpha|$ and $\int_{G_\alpha} x \ d\nu_\alpha \geq \int_{F_\alpha} S(x) \ d\nu_\alpha$. Using this observation, for every $F\in \mathcal{F}$ one can find $G\in \mathcal{F}$ such that $\sum_{\langle n,\alpha\rangle \in G} |x(n, \alpha)| \geq
\sum_{\langle n,
\alpha\rangle \in F} |S(x)(n, \alpha)|$ and so $\| S(x) \|_\mathcal{F} \leq \| x \|_\mathcal{F}$.
\end{proof}

For a Banach space with unconditional basis which is not necessarily monotone, we cannot hope for an isometric embedding, but we using basically the same proof we can show that is can be embedded isomorphically in a combinatorial Banach space.

\begin{thm}\label{complemented} For every Banach space $X$ with a unconditional basis of size $\kappa$ there is a hereditary family $\mathcal{F}$ on $\kappa$ such that $X_\mathcal{F}$ contains a complemented isomorphic copy of $X$. 
\end{thm}

The above theorem says that every Banach space with monotone unconditional basis can be embedded into a combinatorial space. Now, we will show that, under certain assumptions, for a given $\kappa$ there is a combinatorial space containing copies of all Banach
spaces with unconditional bases of size $\kappa$. Moreover, these copies are complemented.

Recall the definition of a universal hypergraph.

\begin{df}
We say that a family $\ef$ of finite subsets of $\kappa$ is a \emph{$\kappa$-universal hypergraph} if for every family $\G$ of finite subsets of $\kappa$ there exists a mapping $\phi:\kappa \xrightarrow[]{} \kappa$ with \[
	A \in \G \iff \phi[A] \in \mathcal{H}.
    \]  
\end{df}

\begin{thm} Suppose that there is a $\kappa$-universal hypergraph. Then there is a hereditary family $\mathcal{F}$ of finite subsets of $\kappa$ such that $X_\mathcal{F}$ contains a complemented copy of
	every Banach space with unconditional basis of size $\kappa$.
\end{thm}
\begin{proof}
	Let $\mathcal{H}$ be a $\kappa$-universal hypergraph and let $\mathcal{F}$ be its hereditary closure (i.e. $\mathcal{F}$ contains all the subsets of elements of $\mathcal{H}$). Then for every hereditary family $\mathcal{G}$ on $\kappa$ there is $\phi \colon \kappa \to \kappa$ such that $\{A\colon \phi[A] \in
	\mathcal{H}\} = \mathcal{G}$. It follows that $\{\phi(\alpha)\colon \alpha<\kappa\}$ generates a complemented copy of $X_\mathcal{G}$ and so $X_\mathcal{F}$ contains a complemented copy of every combinatorial Banach space with density $\kappa$.
	By Theorem \ref{complemented}, in fact
$X_\mathcal{F}$ contains a complemented copy of every Banach space with unconditional basis of size $\kappa$
\end{proof}

It is well-known that there is an $\omega$-universal hypergraph $\mathcal{H}_\omega$ (so called, random hypergraph, which can be realized e.g. as a Fraisse limit of finite hypergraphs).  The space $X_\mathcal{F}$ for $\mathcal{F}$ being the
hereditary closure of $\mathcal{H}_\omega$ is isomorphic to $\omega$-Pe\l czy\'nski space (see \cite{ZOO}). The existence of universal hypergraphs for other infinite cardinals may depend on the axioms. In general, for every infinite

$\kappa$ it is consistent that there is a $\kappa$-universal hypergraph and so that there is a $\kappa$-version of Pe\l czy\'nski space.

There is a rich literature concerning the existence of universal structures (see e.g. \cite{Shelah88}), although in the context of universality, graphs rather than hypergraphs have been studied. The following proposition follows from a general theory
of structures with an amalgamation
property. We enclose a proof as it is short, and perhaps new.

\begin{prop}
    Suppose that $\kappa$ is a cardinal satisfying $\forall \lambda < \kappa \; 2^\lambda \leq \kappa$. Then there exists a $\kappa$-universal hypergraph.
\end{prop}
\begin{proof}
    Let \[
    S = \bigcup_{\alpha < \kappa} \kappa \times \mathcal{P}([\alpha + 1]^{<\omega}).\]
	By assumption on $\kappa$ we conclude that $|S| = \kappa$, so we may construct the $\kappa$-universal hypergraph on $S$.
 Define a family $\mathcal{H}$ of finite subsets of $S$ in the following way \[
	 \{\langle \alpha_1, F_1\rangle, ..., \langle \alpha_n, F_n\rangle \} \in \mathcal{H} \iff \{\alpha_1,...,\alpha_n\} \in F_n,
    \]
where $\alpha_1 < ... < \alpha_n$.

We claim that $\mathcal{H}$ is a $\kappa$-universal hypergraph. Indeed, fix $\G \subs [\kappa]^{<\omega}$. Let \[
		\phi_\G (\alpha) = (\alpha, \G \cap \mathcal{P}([\alpha + 1]^{<\omega}).
    \]Let $A = \{\alpha_1 < ... < \alpha_n\}$. Then \begin{align*}
		&\phi_\G[A] \in \mathcal{H} \iff \{(\alpha_1, \G \cap \mathcal{P}([\alpha_1 + 1]^{<\omega}), ..., (\alpha_n, \G \cap \mathcal{P}([\alpha_n + 1]^{<\omega})\} \in \mathcal{H} \iff \\
				& \iff \{\alpha_1, ..., \alpha_n\} \in \G \cap \mathcal{P}([\alpha_n + 1]^{<\omega}) \iff \{\alpha_1, ..., \alpha_n \} \in \G . 
    \end{align*}
\end{proof}
\begin{cor}
	If CH holds, then there exists a $\omega_1$-universal hypergraph and, more generally, if GCH holds, then for every $\kappa$ there exists a $\kappa$-universal hypergraph.
	Also, in ZFC, there exists an uncountable cardinal $\kappa$ with a $\kappa$-universal hypergraph.
\end{cor}
    
CH is not necessary for the existence of $\omega_1$-universal hypergraph (see e.g. \cite{Mekler90}). However, it is consistent that there are no $\omega_1$-universal hypergraph. It does not imply automatically that consistently
there is no $\omega_1$-Pe\l czy\'nski space but, in the next section, we show that this is indeed the case.

\section{Generic families and the universal space problem}

We will start with defining probably the most natural forcing for adding a hereditary family of finite subsets of $\omega_1$. 

Let $\mathbb{P}$ consists of pairs $p = (F_p,G_p)$, where $F_p\subseteq [\kappa]^{<\omega}$ is hereditary and finite, $G_p\subseteq [\kappa]^{<\omega}$ if finite, and $F_p \cap G_p \subseteq [\kappa]^{\leq 1}$. Let $q\leq p$ if $F_p \subseteq F_q$ and $G_p \subseteq G_q$. The forcing $\mathbb{P}$ introduces
generically a hereditary family $\mathcal{F}$ (simply $\mathcal{F} = \bigcup_{p\in G} F_p$, where $G$ is a $\mathbb{P}$-generic filter) which covers $\kappa$. The families $G_p$ prevents $\mathcal{F}$ to be the whole of $[\kappa]^{<\omega}$. 

For $p\in \mathbb{P}$ let $\mathrm{supp}(p) = \bigcup F_p \cup \bigcup G_p$ and for compatible $p,q$ let $p\cup q = (F_p \cup F_q, G_p \cup G_q)$.

\begin{prop} $\mathbb{P}$ is ccc.
\end{prop}
\begin{proof} Let $(p_\alpha)_{\alpha<\omega_1}$ be an uncountable subfamily of $\mathbb{P}$ and denote $F_\alpha = F_{p_\alpha}$, $G_\alpha = G_{p_\alpha}$. Using
	$\Delta$-system lemma we may assume that the family $(\mathrm{supp}(p_\alpha))_\alpha$ forms a $\Delta$-system with a root $\Delta$ and there is $F'$, $G'$ such that $F_\alpha|\Delta = F'$ and $G_\alpha|\Delta 	= G'$ for each $\alpha$.
	For $\alpha, \beta < \omega_1$, $p_\alpha$ and $p_\beta$ are compatible: it follows from the fact that $|F' \cap G'|\leq 1 $ and the supports of $p_\alpha$, $p_\beta$ are disjoint outside of $\Delta$.
\end{proof}

So, using $\mathbb{P}$ we may add new hereditary family $\mathcal{F}$ on $\kappa$, without collapsing any cardinals. 

Note that if $T$ is an isomorphism, then there is an integer $K$ such that $1/K \| x\| < \|T(x)\| < K \| x\|$.

\begin{thm}\label{generic} Let $\dot{\mathcal{F}}$ be the generic family on $\omega_1$ added by forcing $\mathbb{P}$ and let $\mathcal{H}$ be a hereditary family on $\omega_1$ from the ground model. The space $X_{\dot{\mathcal{F}}}$ cannot be embedded into $X_\mathcal{H}$.
\end{thm}
\begin{proof}
	Suppose that there is $\dot{T}$, a $\mathbb{P}$-name for a function $\dot{T}\colon X_{\dot{\mathcal{F}}} \to X_\mathcal{H}$ such that $1\Vdash$``$\dot{T}$ is an isomorphic embedding''. Denote $\dot{x}_\alpha = \dot{T}(e_\alpha)$ for each $\alpha<\omega_1$. 
	For $k>1$ we say that a condition $p$ $k$-decides $\dot{x}_\alpha$ if there is an interval $I$ of length $1/k$ such that $p\Vdash \| \dot{x}_\alpha \|_{\mathcal{H}} \in I$. 

	Let $p\in \mathbb{P}$. Passing to a stronger condition if needed we may assume that there is $K\in \omega$, $K>1$, such that $p\Vdash 1/K \| x\| < \|T(x)\| < K \| x\|$.
	
For each $\alpha$ fix $p_\alpha \leq p$ such that $p_\alpha$ $K^3$-decides $x_\alpha$. There is an uncountable $\Gamma \subseteq \omega_1$ such that $\mathrm{supp}(p_\alpha)_{\alpha\in \Gamma}$ form a $\Delta$-system with a root $\Delta$ and that
$p_\alpha$'s agree on $\Delta$.

Notice that for each $n$ we can find a set $S\subseteq \Gamma$ such that $|S| = n$ which is free, i.e. for $\alpha\ne \beta \in S$ we have $\alpha \notin \mathrm{supp}(p_\beta)$. Indeed, without loss of generality $(p_\alpha)_{\alpha \in \Gamma}$ are pairwise
different and so the set $\{\min (\mathrm{supp}(p_\alpha)\setminus \Delta) \colon \alpha \in \Gamma\}$ is unbounded. Using this fact it is easy to find the required $S$. Notice also that $S \cap \Delta = \emptyset$.

Find a free $S\subseteq \Gamma$ of size $K^3$ and let $q = \bigcup_{\alpha \in S} p_\alpha$. Then $q$ is a condition. Also, since $q$ $K^3$-decides all $x_\alpha$ in $S$ we have that $q\Vdash \| \sum_{\alpha\in S} x_\alpha\|_{\mathcal{H}} < K^2$ or
$q\Vdash \| \sum_{\alpha\in S} x_\alpha\|_{\mathcal{H}} > K$:
\medskip

\textbf{Case 1.} $q \Vdash \| \sum_{\alpha \in S} x_\alpha \|_\mathcal{H} < K^2$. Then let $r = (F_q \cup \mathcal{P}(S), G_q)$. Since for every $T\in [S]^2$ we have $T\notin G_q$, $r$ is a condition. But 
\[ r\Vdash \| \sum_{\alpha \in S} e_\alpha \|_{\dot{\mathcal{F}}} = K^3 \mbox{ and } \| \sum_{\alpha \in S} x_\alpha \|_\mathcal{H} < K^2, \]
and so $r\leq p$ forces that $K$ is not appropriate isomorphism constant, a contradiction.
\medskip

\textbf{Case 2.} $q \Vdash \| \sum_{\alpha \in S} x_\alpha \|_\mathcal{H} > K$. Then let $r = (F_q, G_q \cup [S]^2)$. Again, since no element of $[S]^2$ belongs to $F_q$, $r$ is a condition and 
\[ r\Vdash \| \sum_{\alpha \in S} e_\alpha \|_{\dot{\mathcal{F}}} = 1  \mbox{ and } \| \sum_{\alpha \in S} x_\alpha \|_\mathcal{H} > K, \]
which again means that $r\leq p$ forces that $K$ is not a proper constant, a contradiction.
\end{proof}

\begin{thm} Consistently, there is no $\omega_1$-Pe\l czy\'nski space.
\end{thm}

\begin{proof} We consider a finite support iteration of length $\omega_2$ of the forcing $\mathbb{P}$ defined above. Let $G$ be the generic.

	We claim that in $V[G]$ there is no $\omega_1$-Pe\l czy\'nski space. Indeed, suppose that $X$ is such a space. Without loss of generality, by Theorem \ref{complemented}, we may assume that $X=X_\mathcal{H}$ for some hereditary family on
	$\omega_1$. As $\mathcal{H}$ is of cardinality $\omega_1$, it appears at an intermediate step $\alpha<\omega_2$ of the iteration. Let $\mathcal{F}$ be the family added generically at step $\alpha+1$. By Theorem \ref{generic} $X_\mathcal{F}$
	cannot be embedded isomorphically into $X_\mathcal{H}$, a contradiction.
\end{proof}

The Banach spaces introduced generically by $\mathbb{P}$ has another peculiar property. Not only it cannot be embedded into any \emph{old} Banach space isomorphically. Also, it does not contain an isomorphic copy of any \emph{old} Banach space of
uncountable density (with unconditional basis). In particular it does contain neither a copy of $\ell_1(\omega_1)$ nor of $c_0(\omega_1)$. The proof is similar to that of Theorem \ref{generic}.

\begin{thm}\label{generic2} Let $\dot{\mathcal{F}}$ be the generic family on $\omega_1$ added by forcing $\mathbb{P}$ and let $\mathcal{H}$ be a hereditary family on $\omega_1$ from the ground model. The space $X_\mathcal{H}$ cannot be embedded into $X_{\dot{\mathcal{F}}}$.
\end{thm}
\begin{proof}
	We start as in the proof of Theorem \ref{generic2}. We choose $p$ and, without loss of generality, we assume that there is $K$ such that  $p$ forces that for each $x$, $\|x\|_\mathcal{H}=1$, we have $1/K < \dot{T}(x) < K$. This time we choose
	$p_\alpha$, for each $\alpha$, so that $p_\alpha \Vdash 1/K < \| \dot{x}_\alpha\|_{\dot{\mathcal{F}}} < K$. Also, this time let $S\subseteq \Gamma$ be a free set of size $K^4$.

\textbf{Case 1.} $\| \sum_{\alpha \in S} e_\alpha \|_\mathcal{H} \leq K^2$.
Let $r = (F_q \cup \mathcal{P}(S), G_q)$. Then
\[ r\Vdash \| \sum_{\alpha \in S} \dot{x}_\alpha \|_{\dot{\mathcal{F}}} > K^4/K = K^3 \mbox{ and } \| \sum_{\alpha \in S} x_\alpha \|_\mathcal{H} \leq K^2, \]
and so $r\leq p$ forces that $K$ is not appropriate isomorphism constant, a contradiction.
\medskip

\textbf{Case 2.} $\| \sum_{\alpha \in S} e_\alpha \|_\mathcal{H} > K^2$.
Then let $r = (F_q, G_q \cup [S]^2)$. We have 
\[ r\Vdash \| \sum_{\alpha \in S} \dot{x}_\alpha \|_{\dot{\mathcal{F}}} < K  \mbox{ and } \| \sum_{\alpha \in S} e_\alpha \|_\mathcal{H} > K^2, \]
and so $r$ forces that $K$ is not a proper constant.
\end{proof}

The forcing $\mathbb{P}$ is equivalent to the forcing adding $\omega_1$ Cohen reals. It can be proved using Koppelberg's characterization theorem, \cite{Koppelberg}, in a simmilar manner e.g. to \cite[Proposition
31]{towers}.

\section{Trees and geometry of balls}\label{sec:Suslin}

In this section we present examples of combinatorial Banach spaces induced by trees. Recall the following definitions:

\begin{df}
Let $r>0$. A subset $A$ of a Banach space $X$ is
\begin{itemize}
\item \emph{$r$-separated} if $\|x-y\|\geq r$ for every distinct $x,y\in A$.
\item \emph{$r$-equilateral} if $\|x-y\|=r$ for every distinct $x,y\in A$.
\item \emph{equilateral} if there is $s>0$ such that $A$ is $s$-equilateral.
\end{itemize}
\end{df}

Clearly every $r$-equilateral set is $r$-separated. The unit ball of every infinite-dimensional Banach space contains an infinite $(1+\varepsilon)$-separated set for some $\varepsilon>0$ (\cite{EltonOdell}), and one may pick one $\varepsilon$ working
for every non-reflexive space (\cite{KryczkaPrus}). There are examples of nonseparable Banach spaces which do not admit uncountable $(1+\varepsilon$)-separated sets (e.g. $c_0(\omega_1)$), but they cannot be reflexive (\cite{HajekKaniaRusso}).

As for equilateral sets, Koszmider and Wark (\cite{KoszmiderWark}) constructed a renorming of $\ell_1(\mathbb R)$ without infinite equilateral sets and Koszmider and Ryduchowski (\cite{KamilKoszmider}) constructed a consistent renorming of $\ell_2(\omega_1)$ without uncountable equilateral sets.

For a broad overview of the topic, see \cite{KoszmiderGeometry}.

If a Banach space $X$ contains an equilateral set of an infinite size $\kappa$, it contains a $1$-equilateral set of size $\kappa$ in its unit sphere, as shown by scaling the set to be $1$-equilateral, moving one of its points to $0$ and forgetting about it.

In this section we work primarily with $X_\F$ where $\F$ is the family of finite chains in a $\omega_1$-tree $T$ on $\omega_1$ (by $\kappa$-tree we understand a tree of height $\kappa$). We show that that the unit ball of such a space $X_\F$ contains an uncountable $(1+\varepsilon)$-separated set if and only if the tree $T$ is not Aronszajn and it admits an uncountable equilateral set if and only if $T$ is not Suslin. 

To this end, we shall first prove a proposition which states roughly that if $T$ is an Aronszajn tree and $(F_\alpha)_{\alpha<\omega_1}$ is a family of pairwise disjoint finite antichains in $T$, then there are $\alpha<\beta<\omega_1$ such that $F_\alpha\cup F_\beta$ is an antichain.

The notation we use is standard: for $s\in T$ by $ht(s)$ we denote the height of $s$ in $T$, i.e. the order type of $\{t\in T:\ t<_T s\}$. $T_\alpha = \{t\in T:\ ht(t)=\alpha\}$ and $T_{\geq \alpha} = \{t\in T:\ ht(t)\geq \alpha\}$. For $\alpha<\omega_1$ by $p_\alpha\colon T_{\geq \alpha} \to T_\alpha$ we denote the natural projection, i.e. for $t\in T_{\geq \alpha}$ $p_\alpha(t)$ is the only element of $T_\alpha$ satisfying $p_\alpha(t)\leq t$. Below we list some properties of $p_\alpha$:

\begin{lem}
Let $\alpha\leq \beta < \omega_1$. Then 
\begin{enumerate}[(1)]
\item $p_\alpha \circ p_\beta = p_\alpha|T_{\geq \beta}$.
\item if $s,t\in T_{\geq \alpha}$ and $p_\alpha(s)\neq p_\alpha(t)$, then $s$ and $t$ are incomparable.
\end{enumerate}
\end{lem}

It is known that every $\omega_1$-tree with all levels finite contains a cofinal branch. Here we reformulate a direct consequence of this fact: 
\begin{lem}\label{fin-levels-lemma}
    Let $T$ be an $\omega_1$-tree and let $S$ be its uncountable subset. If $p_\alpha[S_{\geq \alpha}]$ is finite for cofinally many $\alpha$'s then $T$ contains a cofinal branch.
\end{lem}

\begin{df}
	Let $T$ be a tree. A pairwise disjoint family $\F\subseteq [T]^{<\omega}$ is \emph{narrow} if for every $A,B\in \F$ there are comparable $a\in A$, $b\in B$. If, additionally, $|A|\leq n$ for each $A\in \F$, then we say that $\F$ is \emph{$(n$-$)$narrow}. 
\end{df}

Our goal is to show that in an Aronszajn tree $T$ there are no uncountable narrow families of antichains. In fact, it is equivalent to proving that there are no uncountable narrow families at all:

\begin{lem}\label{n-proper}
Let $T$ be an $\omega_1$-tree such that there exist an uncountable narrow family $\{F_\alpha:\ \alpha<\omega_1\}\subseteq [T]^{\leq n}$. Then there exists an uncountable $n$-narrow family $\{G_\alpha:\ \alpha<\omega_1\}$ with $G_\alpha\subseteq
T_\alpha$. 
\end{lem}
\begin{proof}
By passing to an uncountable subfamily we may assume that for every $\alpha<\beta<\omega_1$ we have $\max ht[F_\alpha] < \min ht[F_\beta]$ and $\min ht[F_0]>0$; in particular $\min ht[F_\alpha]\geq \alpha$ for every $\alpha<\omega_1$. Let $G_\alpha = p_\alpha[F_\alpha]$ for $\alpha<\omega_1$. Clearly $G_\alpha \subseteq T_\alpha$ and $|G_\alpha|\leq |F_\alpha|\leq n$ for $\alpha<\omega_1$, so it is enough to check that $\{G_\alpha:\ \alpha<\omega_1\}$ is narrow. To this end fix $\alpha<\beta<\omega_1$ and pick $a\in F_\alpha$, $b\in F_\beta$ such that $a<b$. Then $a$ and $p_\beta(b)$ are comparable, and therefore $p_\alpha(a)<p_\beta(b)$, which concludes the proof.
\end{proof}

We say that a family $\F$ has the \emph{intersection property} if for every $A,B\in \F$ the intersection $A\cap B$ is nonempty (such families are also often called \emph{linked families}). 

The following lemma lets us reinterpret the narrowness of a family.
\begin{lem}\label{eqivalent-refolmulation}
Let $T$ be an $\omega_1$-tree and let $\F\subseteq [T]^{<\omega}$ be an uncountable family with $\min(ht[A]) = \max(ht[A])$ for every $A\in \F$. Then the following are equivalent:\begin{enumerate}
        \item $\F$ is narrow,
        \item for every $\alpha<\omega_1$ the family $\{p_\alpha[A]: A\in \F, \min(ht[A]) \geq \alpha \}$ has the intersection property.
    \end{enumerate}
\end{lem}
\begin{proof}
	Assume that $\F$ is narrow. Pick $\alpha<\omega_1$ and $A, B\in \F$ such that $ht[A], ht[B]\geq \alpha$. There are $a\in A, b\in B$, which are comparable, and so $p_\alpha(a) = p_\alpha(b)$. Thus, $p_\alpha[A] \cap p_\alpha[B]\neq \emptyset$. 
	Now assume that for every $\alpha<\omega_1$ the family $\{p_\alpha[A]: A\in \F, \min(ht[A]) \geq \alpha \}$ has the intersection property. Pick $A, B\in \F$. We may assume that $A\subs T_\alpha, B\subs T_\beta$ and $\alpha < \beta$. Then there exists $a \in p_\alpha[A] \cap p_\alpha[B] = A \cap p_\alpha[B]$. So there exists $b \in B$ so that $a = p_\alpha(b)$, but this means that $a < b$.
\end{proof}

For a family $\mathcal{B} \subseteq \mathcal{P}(X)$, say that $A\subseteq X$ is a \emph{transversal} if $A\cap B \ne \emptyset$ for each $B\in \mathcal{B}$. Note that a family $\mathcal{B}$ has the intersection property if and only if each element of
$\mathcal{B}$ is a transversal of $\mathcal{B}$.

\begin{lem} \label{transversal-lemma}
	Let $\mathcal{A}$ be a family of finite sets and let $\mathcal{B}$ be a family consisting of finite transversals of $\mathcal{A}$ with size at most $k$. Then either $\mathcal{B}$ is finite or $\mathcal{A}$ has a transversal of size strictly less than $k$.
\end{lem}
\begin{proof}
	Assume that $\mathcal{B}$ is infinite and let $\mathcal{B} = \{B_i\colon i\in \omega\}$. By the Sunflower Lemma (\cite{ErdosRado1960}) we may assume that there is $R$ such that $B_i \cap B_j = R$ for $i\neq j$. Now note that $|R| < k$ (as
	$(B_i)_i$ are pairwise different) and $R$ is a transversal: if a set $A \in \mathcal{A}$ does not intersect $R$, then it could intersect at most finitely many sets $B_i$.
\end{proof}

\begin{prop}\label{branch}
    Suppose that $T$ is an $\omega_1$-tree that contains a narrow family. Then $T$ has a cofinal branch.
\end{prop}
\begin{proof}
	Assume that $T$ contains a narrow family $\ef = \{F_\alpha\colon \alpha<\omega_1\}$. We may assume that $|F_\alpha| = n$ and for some $n$ and each $\alpha$ and, thanks to Lemma \ref{n-proper}, that $F_\alpha \subs T_\alpha.$ Let $m$ be the least
	natural number such that there exists an unbounded set $I\subs \omega_1$ and a family $\{ S_\alpha: \alpha\in I\}$ such that for each $\alpha \in I$: \begin{enumerate}
	    \item $|S_\alpha| \leq m,$
        \item $S_\alpha \subs T_\alpha$,
        \item $\exists \beta\; \forall \xi > \beta\; p_\alpha [F_\xi] \cap S_\alpha \neq \emptyset.$ (Denote such $\beta$ as $\beta(\alpha).$)
	\end{enumerate}
	It follows from Lemma \ref{eqivalent-refolmulation} that $m \leq n.$ Let $I$ and $\{S_\alpha: \alpha\in I\}$ be as above, and fix $\alpha\in I$. There exists an index $\beta'$ such that for every $\eta \in I \bs \alpha$ and every $\xi > \beta'$
	we have $p_\alpha[S_\eta] \cap p_\alpha[F_\xi]\neq \emptyset.$ Indeed, it is enough to take $\beta'<\omega_1$ bigger than $\min(\beta(\gamma)\colon p_\alpha(F_\gamma) = p_\alpha[S_\eta\})$ for any $\eta\in I\setminus \alpha$ (it is possible as there are
		only countably many possibilites for $p_\alpha[S_\eta]$). By Lemma \ref{transversal-lemma}, if $\bigcup_{\eta\in I\bs \alpha} p_\alpha[S_\eta]$ is infinite, then there exists $S'_\alpha$ with cardinality smaller than $m$ that satisfies (2) and 
		(3). Hence, because of the choice of $m$, there
exists $\gamma$ such that for all $\alpha > \gamma$ we have $|\bigcup_{\eta\in I\bs \alpha} p_\alpha[S_\eta]| < \omega$. By Lemma \ref{fin-levels-lemma} it means that there is a cofinal branch in $T$.
\end{proof}

Piotr Koszmider remarked that the above proposition was proved, by different methods, in \cite[Theorem 8.2]{Baumgartner}.

Now we are ready to proceed to the geometric properties of $X_\F$ stated at the beginning of this section. Let us, however, add a suspense by considering a digression about finite combinatorics.

\begin{rem}It might be an interesting question in finite combinatorics to find, for a given $n$, a maximal size $t(n)$ of $\bigcup \mathcal{F}$ such that $\mathcal{F} \subseteq [\omega]^n$ has the intersection property and does not have transversal
	of size $n-1$ (so, to prove a quantitive version of Lemma \ref{transversal-lemma}).
Notice that if there is a projective plane $P$ of order $n-1$ and $\mathcal{F}$ is the set of all the lines on $P$, then $\mathcal{F}\subseteq [\omega]^n$, it has the intersection property, and it does not have a transversal of size $n-1$. Hence, in
this case, $t(n)$ is at least of the size of $P$ and so $t(n) \geq (n-1)^2 + n$. In fact, we conjecture that $t(n) = (n-1)^2 + n$ for every $n$.
\end{rem}

Now we proceed to the matters of existence of uncountable equilateral and separated subsets of the unit ball of $X_\F$, where $\F$ is the family of finite chains in a tree $T$ on $\omega_1$.

\begin{lem}\label{reduction}
Let $(x_\alpha)_{\alpha<\omega_1}$ be a family of vectors in the unit ball of a Banach space $X_\F$. Then for every $\delta>0$ there is an uncountable $\Gamma\subseteq \omega_1$ and vectors $(z_\alpha)_{\alpha\in \Gamma}$ with finite, pairwise disjoints supports and rational values in the unit ball of $X_\F$ such that for every $\alpha<\beta$ in $\Gamma$

$$\|x_\alpha-x_\beta\| - \delta \leq \|z_\alpha-z_\beta\| \leq \|x_\alpha-x_\beta\|+\delta.$$
\end{lem}
\begin{proof}
	For $\alpha<\omega_1$ let $y_\alpha$ be a vector with finite support and rational values such that $\|x_\alpha-y_\alpha\|\leq \frac{\delta}{2}.$ Then, clearly, for every $\alpha<\beta<\omega_1$
$$\|x_\alpha-x_\beta\| - \delta \leq \|y_\alpha-y_\beta\| \leq \|x_\alpha-x_\beta\|+\delta.$$

Now pass to an uncountable set $\Gamma\subseteq \omega_1$ such that $(\supp(y_\alpha))_{\alpha<\omega_1}$ is a $\Delta$-system with the root $\Delta$ and $y_\alpha|\Delta = y_\beta|\Delta$ for every $\alpha<\beta$ in $\Gamma$, which is possible since
$y_\alpha$'s have rational values. For $\alpha\in \Gamma$ let $z_\alpha = P_{\omega_1\setminus \Delta} (y_\alpha)$. Then for every $\alpha<\beta$ in $\Gamma$ we have
$$\|z_\alpha-z_\beta\| = \|y_\alpha-y_\beta\|,$$
which completes the proof.

\end{proof}

\begin{thm}\label{Aronszajn}
Let $T$ be an $\omega_1$-tree and let $\F$ be the family of finite chains in $T$. Then the following conditions are equivalent:
\begin{enumerate}[(1)]
	\item $T$ is Aronszajn,
\item the unit ball of the space $X_\F$ does not contain an uncountable $(1+\varepsilon)$-separated sets for any $\varepsilon>0$.
\end{enumerate}
\end{thm}
\begin{proof}
If $T$ contains a cofinal branch $B\subseteq T$, then clearly $\{\chi_{\{b\}}:\ b\in B\}$ is an uncountable $2$-separated set in the unit ball of $X_\F$. Now suppose that $(x_\alpha)_{\alpha<\omega_1}$ is an uncountable $(1+\varepsilon)$-separated
set in the unit ball of $X_\F$ for some $\varepsilon>0$. By Lemma \ref{reduction} we may assume, picking a smaller $\varepsilon>0$ if needed, that $x_\alpha$'s have finite, pairwise disjoint supports and rational values. Also, we may assume that there is $n$
such that $|\supp(x_\alpha)|=n$ for every $\alpha<\omega_1$. For every $\alpha<\beta<\omega_1$ $$\|x_\alpha-x_\beta\|\geq 1+\varepsilon > 1 \geq \max\{\|x_\alpha\|,\|x_\beta\|\},$$
which means that there are comparable $a\in \supp(x_\alpha)$, $b\in \supp(x_\beta)$, i.e. the family $(\supp(x_\alpha))_{\alpha<\omega_1}$ is narrow. Therefore, by Proposition \ref{branch} the tree  $T$ contains a cofinal branch.
\end{proof}

\begin{thm}\label{Suslin-equilateral}
Let $T$ be an $\omega_1$-tree and let $\F$ be the family of finite chains in $T$. Then the following conditions are equivalent:
\begin{enumerate}[(1)]
	\item $T$ is Suslin,
\item $X_\F$ does not contain an uncountable equilateral set
\end{enumerate}
\end{thm}
\begin{proof}
(2) $\implies$ (1).  Note that if $A\subseteq T$ is an uncountable chain or antichain, then $\{\chi_{\{a\}}:\ a\in A\}$ is an uncountable equilateral set.\medskip

(1)$\implies$ (2). Assume $T$ is a Suslin tree and $X_\F$ contains an uncountable equilateral set. As discussed beforehand, if there is an uncountable equilateral set in $X_\F$, then there is an uncountable $1$-equilateral set $\{x_\alpha:\
\alpha<\omega_1\}$ in the unit sphere of $X_\F$. Apply Lemma \ref{reduction} and pass to an uncountable subfamily to find vectors $(z_\alpha)_{\alpha<\omega_1}$ with finite, pairwise disjoint supports and rational values such that $$0.98\leq
\|z_\alpha-z_\beta\|\leq 1.02$$ for every $\alpha<\omega_1$. We may assume that $r=\|z_\alpha\|$ does not depend on $\alpha<\omega_1$. Let $F_\alpha\subseteq \supp(z_\alpha)$ be a chain such that 
$$r = \sum_{a\in F_\alpha} |z_\alpha(a)|$$
and let $t_\alpha=\max F_\alpha$. Since $T$ is Suslin, the set $\{t_\alpha:\ \alpha<\omega_1\}$ cannot be an antichain, and therefore there are $\alpha<\beta<\omega_1$ such that $t_\alpha$ and $t_\beta$ are comparable, i.e. $F_\alpha \cup F_\beta$ is a chain. For such $\alpha,\beta$ we have 
$$\|z_\alpha-z_\beta\| = 2r\leq 1.02,$$
i.e. $r\leq 0.56$.
However, by Theorem \ref{Aronszajn}, the unit ball of $X_\F$ does not contain an uncountable $(1+\frac{0.38}{0.6})$-separated sets, and so there are $\alpha<\beta<\omega_1$ such that
$$0.98 \leq \|z_\alpha-z_\beta\| < (1+\frac{0.38}{0.6})r,$$
which means that $r>0.6$ -- a contradiction.
\end{proof}

Let us formulate an obvioius corollary.

\begin{cor}\label{Suslin} If $G$ is a comparability graph of a Suslin tree, then $X_G$ does not contain an uncountable equilateral set.
\end{cor}

The following proposition, being a direct generalization of a Proposition 19 in \cite{KamilKoszmider}, shows that some additional set-theoretic requirements are necessary to obtain such a space.

\begin{prop}[MA+$\neg$ CH]\label{MA-equilateral}
Let $\F\subseteq [\omega_1]^{<\omega}$. Then the unit sphere of the space $X_\F$ contains an uncountable equilateral set.
\end{prop}
\begin{proof}
	Consider the forcing notion $\mathbb P=\F^\perp$ with the order given by reverse inclusion. Clearly, for every $A\in \mathbb P$ we have $\|\chi_A\|=1$.
	Under MA + $\neg$ CH the forcing $\mathbb P$ either contains an uncountable antichain or may be represented as $\mathbb P = \bigcup_{n<\omega} \mathcal A_n$ with each $\mathcal A_n$ closed under finite unions.

If $\mathcal A\subseteq \mathbb P$ is an uncountable antichain, then $\{ \chi_A:\ A\in \mathcal A\}$ is an uncountable $2$-equilateral set. If $\mathbb P = \bigcup_{n<\omega} \mathcal A_n$ with each $\mathcal A_n$ closed under finite unions, then
there is $n<\omega$ and an uncountable $I\subseteq \omega_1$ such that $\{\alpha\}\in A_n$ for every $\alpha\in I$. It follows that $I$ is in the closure of $\mathcal{F}^\perp$, and so $\mathcal{F}|I$ consists of singletons. Hence, $X_\F$ contains an isometric copy of $c_0(I)$ , and therefore an uncountable $1$-equilateral set $\{\chi_{\{\alpha\}}:\ \alpha\in I\}$. 
\end{proof}

If, instead of $\F$ being the family of finite chains in $T$, we consider it to be the family of finite antichains, the situation turns out to be drastically different.

\begin{prop}
Let $\kappa$ be an infinite regular cardinal, let $T$ be a well pruned $\kappa$-Aronszajn tree. Let $\mathcal{F}$ be the family of finite antichains in $T$.
Then the unit sphere of the Banach space $X_\mathcal F$ contains a $(1+\frac{1}{2})$-separated set of size $\kappa$.
\end{prop}
\begin{proof}
We inductively construct sequences $(a_\alpha,b_\alpha,c_\alpha)_{\alpha<\kappa}$ of nodes in $T$ such that:
\begin{itemize}
\item $c_\alpha<a_\alpha,b_\alpha$ for every $\alpha<\kappa$
\item $a_\alpha$ and $b_\alpha$ are incomparable for every $\alpha<\kappa$
\item $\max\{ht(a_\alpha), ht(b_\alpha)\} < ht(c_\beta)$ for every $\alpha<\beta<\kappa$
\end{itemize}

We claim that $(\frac{1}{2}\chi_{\{a_\alpha,b_\alpha\}})_{\alpha<\omega_1}$ form a $(1+\frac{1}{2})$-equilateral set in the unit sphere of $X_\F$.

Indeed, fix $\alpha<\beta<\kappa$. It cannot be the case that both $a_\beta > a_\alpha$ and $b_\beta>b_\alpha$, since we would have that $c_\beta$ is above $a_\alpha$ and $b_\alpha$ at the same time. For the same reason it cannot be the case that both $a_\beta>b_\alpha$ and $b_\beta>a_\alpha$. Therefore the set $\{a_\alpha, b_\alpha, a_\beta, b_\beta\}$ contains an antichain of size 3 and so
$$\|\chi_{\{a_\alpha,b_\alpha\}}- \chi_{\{a_\beta,b_\beta\}}\|\geq 3.$$

\end{proof}

In the light of Corollary \ref{Suslin} the following question appears: is the existence of a combinatorial Banach space of density $\omega_1$ without uncountable equilateral set in the unit ball equivalent to the existence of a Suslin tree? The
answer is negative: the construction from \cite{KamilKoszmider} yielding such a space works under CH (and CH does not imply the existence of a Suslin tree). 

\begin{prob} Is there a (natural) property $\varphi$ of a Banach space such that the existence of a Banach space with property $\varphi$ is equivalent to the existence of a Suslin tree?
\end{prob}

Recall that the Suslin hypothesis, saying that Suslin trees do not exist, has a topological reformulation: there is no linearly ordered compact space which is ccc and non-separable. In the above problem we ask about a reformulation of Suslin
hypothesis in the terms of Banach spaces. Of course we can define such $\varphi$ in the following way: a Banach space has $\varphi$ if it is induced by chains of a tree on $\omega_1$ and does not contain neither copy of $\ell_1(\omega_1)$ nor a copy
of $c_0(\omega_1)$. However, such $\varphi$ is highly unsatisfactory. We seek for a reformulation which does not refer directly to any combinatorial property. 

There is also another problem which seems to be natural. Say that a Banach space is \emph{Suslin} if it is of the form $X_G$, where $G$ is the comparability graph of a Suslin tree. Recall that whenever $T$ is a Suslin tree, and $S$ is its subtree of height $\omega_1$, $S$ is also a Suslin tree. The following asks if an analogous property is true at the \emph{analytic} level:

\begin{prob} Let $X$ be a Suslin space, and let $Y$ be a closed subset of $X$ of density $\omega_1$. Does $Y$ contains an isomorphic copy of a Suslin space? 
\end{prob}

One can consider a particular case when $Y$ is generated by a sequence $(x_\alpha)_{\alpha<\omega_1}$, where $\|x_\alpha\| = 1$, and $x_\alpha$ is supported on $T_\alpha$, for each $\alpha$. For example, if each $x_\alpha$ is supported by a
singleton, then $Y$ is a Suslin space itself, according to the remark above. However, even the case where each $x_\alpha$ is supported by two points seems to be unclear and interesting: let $l_\alpha, r_\alpha \in T_\alpha$ for each $\alpha$ and let
$x_\alpha = e_{l_\alpha} + e_{r_\alpha}$. The question is if the space $Y$ generated by $(x_\alpha)_{\alpha<\omega_1}$ contains an isomorphic copy of a Suslin space. Both $\{l_\alpha\colon \alpha<\omega_1\}$ and $\{r_\alpha\colon \alpha<\omega_1\}$
are Suslin trees, so the question is if we can \emph{synchronize} them, e.g. by passing to an uncountable $I\subseteq \omega_1$, in the way that $(x_\alpha)_{\alpha\in I}$ generates a Suslin space.

\section{AI disclosure}

The sole contribution of AI agents to this work was the generation of this very sentence.

\bibliographystyle{alpha}
\bibliography{bib-ppoint}

\begin{thebibliography}{BNFJPB25}

\bibitem[AM93]{Argyros93}
Spiros Argyros and Sofokles Mercourakis.
\newblock On weakly {L}indelof {B}anach spaces.
\newblock {\em The Rocky Mountain Journal of Mathematics}, 23(2):395--446,
  1993.

\bibitem[AM14]{LargeArgyros}
Spiros~A. Argyros and Pavlos Motakis.
\newblock {$\alpha$}-large families and applications to {B}anach space theory.
\newblock {\em Topology Appl.}, 172:47--67, 2014.

\bibitem[Bau83]{Baumgartner}
James~E. Baumgartner.
\newblock Iterated forcing.
\newblock In {\em Surveys in Set Theory}, volume~87 of {\em London Mathematical
  Society Lecture Note Series}, pages 1--59. Cambridge University Press,
  Cambridge, 1983.

\bibitem[BNC15]{towers}
Piotr Borodulin-Nadzieja and David Chodounsk{\'y}.
\newblock Hausdorff gaps and towers in $\mathcal{P}(\omega)/\mathrm{Fin}$.
\newblock {\em Fundamenta Mathematicae}, 229(3):197--229, 2015.

\bibitem[BNFA25]{Perfect}
Piotr Borodulin-Nadzieja, Barnabas Farkas, and Pelczar-Barwacz Anna.
\newblock Perfect graphs and banach spaces.
\newblock {\em preprint}, 2025.

\bibitem[BNFJPB25]{ZOO}
Piotr Borodulin-Nadzieja, Barnab\'as Farkas, Sebastian Jachimek, and Anna
  Pelczar-Barwacz.
\newblock The zoo of combinatorial {B}anach spaces.
\newblock {\em Studia Math.}, 282(2):101--132, 2025.

\bibitem[Dre88]{Drewnowski}
Lech Drewnowski.
\newblock On uncountable unconditional bases in {B}anach spaces.
\newblock {\em Studia Math.}, 90(3):191--196, 1988.

\bibitem[EO81]{EltonOdell}
John Elton and Edward Odell.
\newblock The unit ball of every infinite-dimensional normed linear space
  contains a $(1+\varepsilon)$-separated sequence.
\newblock {\em Colloquium Mathematicae}, 44(1):105--109, 1981.

\bibitem[ER60]{ErdosRado1960}
Paul Erd{\H{o}}s and Richard Rado.
\newblock Intersection theorems for systems of sets.
\newblock {\em Journal of the London Mathematical Society}, s1-35(1):85--90,
  1960.

\bibitem[Fre]{Fremlin-DU}
David~H. Fremlin.
\newblock Problem {DU}.
\newblock note available at
  \url{https://www1.essex.ac.uk/maths/people/fremlin/probDU.pdf}.

\bibitem[Gra98]{Granero}
Antonio~S. Granero.
\newblock On the complemented subspaces of {$c_0(I)$}.
\newblock {\em Atti Sem. Mat. Fis. Univ. Modena}, 46(1):35--36, 1998.

\bibitem[HKR20]{HajekKaniaRusso}
Petr H{\'a}jek, Tomasz Kania, and Tommaso Russo.
\newblock Separated sets and auerbach systems in {Banach} spaces.
\newblock {\em Transactions of the American Mathematical Society},
  373(10):6961--6998, 2020.

\bibitem[HM75]{HajnalMate}
Andr{\'a}s Hajnal and Attila M{\'a}t{\'e}.
\newblock Set mappings, partitions, and chromatic numbers.
\newblock In H.~E. Rose and J.~C. Shepherdson, editors, {\em Logic Colloquium
  '73}, volume~80 of {\em Studies in Logic and the Foundations of Mathematics},
  pages 347--379. North-Holland, Amsterdam, 1975.

\bibitem[HR19]{RussoHajek}
Petr H\'ajek and Tommaso Russo.
\newblock An uncountable version of {P}t\'ak's combinatorial lemma.
\newblock {\em J. Math. Anal. Appl.}, 470(2):1070--1080, 2019.

\bibitem[Kop93]{Koppelberg}
Sabine Koppelberg.
\newblock Characterizations of {C}ohen algebras.
\newblock In {\em Papers on general topology and applications ({M}adison, {WI},
  1991)}, volume 704 of {\em Ann. New York Acad. Sci.}, pages 222--237. New
  York Acad. Sci., New York, 1993.

\bibitem[Kos25]{KoszmiderGeometry}
Piotr Koszmider.
\newblock On ramsey-type properties of the distance in nonseparable spheres.
\newblock {\em Trans. Amer. Math. Soc.}, 378(4):2477--2515, 2025.

\bibitem[KP01]{KryczkaPrus}
Andrzej Kryczka and Stanis{\l}aw Prus.
\newblock Separated sequences in nonreflexive {Banach} spaces.
\newblock {\em Proc. Amer. Math. Soc.}, 129(1):155--163, 2001.

\bibitem[KR24]{KamilKoszmider}
Piotr Koszmider and Kamil Ryduchowski.
\newblock Equilateral and separated sets in some {H}ilbert generated {B}anach
  spaces.
\newblock {\em Proc. Amer. Math. Soc.}, 152(3):1003--1017, 2024.

\bibitem[KvM95]{KunenMill}
Kenneth Kunen and Jan van Mill.
\newblock Measures on {C}orson compact spaces.
\newblock {\em Fund. Math.}, 147(1):61--72, 1995.

\bibitem[KW22]{KoszmiderWark}
Piotr Koszmider and Hugh~M. Wark.
\newblock Large banach spaces with no infinite equilateral sets.
\newblock {\em Bulletin of the London Mathematical Society}, 54(6):2066--2077,
  2022.

\bibitem[LZ69]{Zippin}
Joram Lindenstrauss and Moshe Zippin.
\newblock Banach spaces with a unique unconditional basis.
\newblock {\em J. Functional Analysis}, 3:115--125, 1969.

\bibitem[Mek90]{Mekler90}
Alan~H. Mekler.
\newblock Universal structures in power {$\aleph_1$}.
\newblock {\em J. Symbolic Logic}, 55(2):466--477, 1990.

\bibitem[MPZ25]{Zakrzewski}
Witold Marciszewski, Grzegorz Plebanek, and Krzysztof Zakrzewski.
\newblock Digging into the classes of κ-{C}orson compact spaces.
\newblock {\em Israel J. Math.}, 2025.

\bibitem[Pc68]{Pelczynski68}
Aleksander Pe\l~czy\'nski.
\newblock On {B}anach spaces containing {$L\sb{1}(\mu )$}.
\newblock {\em Studia Math.}, 30:231--246, 1968.

\bibitem[Pe{\l}69]{universalPelczynski}
Aleksander Pe{\l}czy\'{n}ski.
\newblock Universal bases.
\newblock {\em Studia Math.}, 32:247--268, 1969.

\bibitem[Ple97]{Plebanek97}
Grzegorz Plebanek.
\newblock Nonseparable {R}adon measures and small compact spaces.
\newblock {\em Fund. Math.}, 153(1):25--40, 1997.

\bibitem[Ple02]{Plebanek02}
Grzegorz Plebanek.
\newblock On compact spaces carrying random measures of large {M}aharam type.
\newblock volume~43, pages 87--99. 2002.
\newblock 30th Winter School on Abstract Analysis (Lhota nade Rohanovem/Litice
  u \v Cesk\'e{} L\'ipy, 2002).

\bibitem[Ple24]{Plebanek2024}
Grzegorz Plebanek.
\newblock A survey on topological properties of p(k) spaces.
\newblock {\em Japanese Journal of Mathematics}, 19(2):143--180, Oct 2024.

\bibitem[She88]{Shelah88}
Saharon Shelah.
\newblock A graph which embeds all small graphs on any large set of vertices.
\newblock {\em Ann. Pure Appl. Logic}, 38(2):171--183, 1988.

\bibitem[SS88]{Shelah-Steprans}
Saharon Shelah and Juris Stepr\=ans.
\newblock A {B}anach space on which there are few operators.
\newblock {\em Proc. Amer. Math. Soc.}, 104(1):101--105, 1988.

\end{thebibliography}

\end{document}